\documentclass[12pt]{amsart}
\usepackage{amsmath}
\usepackage{amssymb}
\usepackage{amsthm}
\usepackage[msc-links,abbrev]{amsrefs}
\usepackage[noabbrev,capitalize]{cleveref}

\def\sideremark#1{\ifvmode\leavevmode\fi\vadjust{\vbox to0pt{\vss
 \hbox to 0pt{\hskip\hsize\hskip1em
 \vbox{\hsize3cm\tiny\raggedright\pretolerance10000
 \noindent #1\hfill}\hss}\vbox to8pt{\vfil}\vss}}}

\newtheorem{theorem}{Theorem}[section]
\newtheorem{proposition}[theorem]{Proposition}
\newtheorem{lemma}[theorem]{Lemma}
\newtheorem{corollary}[theorem]{Corollary}

 \theoremstyle{definition}
 \newtheorem{definition}[theorem]{Definition}

 \theoremstyle{remark}
 \newtheorem{remark}[theorem]{Remark}

\numberwithin{equation}{section}

\makeatletter
\renewcommand\subsubsection{\@secnumfont}{\bfseries}%
\renewcommand\subsubsection{\@startsection{subsubsection}{3}
\z@{.5\linespacing\@plus.7\linespacing}{-.5em}%
{\normalfont\bfseries}}
\makeatother
\RequirePackage[margin=1.3in]{geometry}
\usepackage{comment}
\usepackage{xcolor}
\usepackage{tensor}
\usepackage[rightcaption]{sidecap}
\usepackage{graphicx} 
\graphicspath{ {images/} }
\usepackage{wrapfig}
\usepackage{float}
\RequirePackage{tcolorbox}

\usepackage{amsmath,amssymb}
\usepackage{xcolor}
\usepackage{tikz}
\usepackage{graphicx}
\usetikzlibrary{arrows.meta,angles,quotes}

\newcommand{\Ric}{\text{Ric}}

\newcommand{\Tr}{\text{tr}}
\newcommand{\Ch}{\text{Ch}}
\newcommand{\C}{\mathcal{C}}
\newcommand{\Sec}{\text{Sec}}

\addtocontents{toc}{\setcounter{tocdepth}{1}} 

\begin{document}

\title[Holography and Cheeger constant]{Holography and Cheeger constant of asymptotically CMC submanifolds}
\author{Samuel P\'erez-Ayala}
\address{Samuel P\'erez-Ayala\\Department of Mathematics \\Princeton University  \\ Princeton\\NJ 08544\\ USA}
\email{sperezayala@princeton.edu}
\author{Aaron J. Tyrrell}
\address{Aaron J. Tyrrell\\Department of Mathematics \\ University of Notre Dame \\ Notre Dame \\ IN 46556 \\ USA}
\email{atyrrell@nd.edu}

\subjclass[2020]{}
\begin{abstract}
Let $(M^{n+1},g_+)$ be an asymptotically hyperbolic manifold. We compute the Cheeger constant of conformally compact asymptotically constant mean curvature submanifolds $ \iota : Y^{k+1} \to (M^{n+1},g_+)$ with arbitrary codimension. As an application, we provide two classes of examples of $(n+1)$-dimensional asymptotically hyperbolic manifolds with Cheeger constant equal to $n$, whose conformal infinity is of the following types: 1) positive Yamabe constant, and 2) negative Yamabe constant. Moreover, in the same spirit as Blitz--Gover--Waldron \cite{BlitzSamuel2021CFFa}, we show that an asymptotically hyperbolic manifold with umbilic boundary is conformally weakly Poincar\'e--Einstein if and only if the third conformal fundamental form of the boundary vanishes. Next, in the space of asymptotically minimal hypersurfaces $Y$ within a Poincar\'e--Einstein manifold, we identify an extrinsic conformal invariant of $\partial Y$ which obstructs the vanishing of the mean curvature of $Y$ to second order. This conformal invariant is a linear combination of two Riemannian hypersurface invariants of $\partial Y,$ one which depends on its extrinsic geometry within $\overline{Y}$ and the other on its extrinsic geometry within $\partial M;$ neither of which are conformal invariants individually. Finally, we show that for asymptotically minimal hypersurfaces with mean curvature vanishing to second order inside of a Poincar\'e--Einstein space, being weakly Poincar\'e--Einstein is equivalent to the boundary of $Y$ having vanishing second and third conformal fundamental forms when viewed as a hypersurface within the conformal infinity.
\end{abstract}

\maketitle

\section{Introduction}

Let $\overline{M}^{n+1}$ be an $(n+1)$-dimensional compact smooth Riemannian manifold with nonempty boundary $\Sigma^n = \partial M$ and interior $M$. We say that a complete and noncompact manifold $(M^{n+1}, g_+)$ is conformally compact (CC) if $\overline g := r^2g_+$ extends to a smooth metric on $\overline{M}$, where $r$ is a defining function for the boundary $\Sigma$. By a defining function for $\Sigma$ we mean that 
\[r>0 \;\text{ in }\; M^{n+1},\; r^{-1}(0)=\Sigma\; \text{ and }\; dr\not = 0\; \text{ on }\; \Sigma.\] More specifically, we say that $(M^{n+1}, g_+)$ is $C^{m,\alpha}$ conformally compact, $\alpha\in(0,1)$, $m\in \mathbb{N}$, if $\bar g = r^2g_+$ extends to a $C^{m,\alpha}$ metric on $\overline{M}^{n+1}$. Throughout this paper, we will always work under the assumption that $g_+$ is a $C^{m,\alpha}$ CC metric with $m\ge 3$ and $\alpha\in(0,1)$. The same regularity, at least $C^{3,\alpha}$, will be required when dealing with submanifolds.

Given a conformally compact manifold $(M^{n+1}, g_+)$, if the sectional curvatures $\Sec(p)$ approach $-1$ as $p$ approaches the boundary $\Sigma$, then we say that $(M^{n+1},g_+)$ is asymptotically hyperbolic (AH). Interestingly, as observed by Mazzeo in \cite{mazzeo1986hodge}, 
conformally transforming the metric $g_+ = r^{-2}\overline g$ gives
\begin{equation}\label{TransLawRiemann}
    (R_{g_+})_{ijkl} = -|dr|^2_{\overline g}[(g_+)_{ik}(g_+)_{jl} - (g_+)_{il}(g_+)_{jk}] + O_{ijkl}(r^{-3}),
\end{equation}
therefore for an AH manifold we necessarily have  $|dr|^2_{\overline g} = 1$ on the boundary. Notice that, even for just conformally compact manifolds, (\ref{TransLawRiemann}) implies that the sectional curvatures of $(M^{n+1},g_+)$ approach $\left(-|dr|^2_{\overline g}\right)|_{\Sigma}$ at the boundary, meaning that the function $\left(-|dr|^2_{\overline g}\right)|_{\Sigma}$ is an invariant of $g_+$. If a conformally compact manifold satisfies $\Ric_{g_+} = -ng_+$, then we call $(M^{n+1},g_+)$ a Poincar\'e--Einstein (PE) manifold\footnote{Poincar\'e--Einstein manifolds are also know in literature as conformally compact Einstein (CCE) manifolds or asymptotically hyperbolic Einstein (AHE) manifolds.}, and it can be shown by contracting (\ref{TransLawRiemann}) that any PE manifold is also AH. In terms of sets, we have:
\[
\{\text{PE}\}\subset \{\text{AH}\} \subset \{\text{CC}\}.
\]

In the context of conformally compact manifolds, the conformal class $[g_+]_\infty:=[\overline g|_{T\Sigma}]$ is also an invariant of $g_+,$ and it is called the conformal infinity of $(M^{n+1},g_+)$. To illustrate some of the terminology we have introduced, we briefly discuss an important example. The model case for a Poincar\'e--Einstein manifold is $\mathbb{H}^{n+1}(-1)$, that is, $(\mathbb{B}^{n+1}, g_H)$ where $g_H=4(1-|x|^2)^{-2}g_{\mathbb{R}^{n+1}}$. The functions $r_1,r_2:\mathbb{B}^{n+1}\rightarrow \mathbb{R}^{n+1}$, defined as $r_1(x) = 1 - |x|$ and $r_2(x) = (1-|x|)(1+|x|)^{-1}$, where $|x|=\sqrt{\sum_{i=1}^{n+1} x_i^2}$, are examples of defining functions for $\partial \mathbb{B}^{n+1} = \mathbb{S}^n$ with both $r_1^2g_H$ and $r_2^2g_H$ extending to metrics on the closure of $\mathbb{B}^{n+1}$. However only $r_2$ is an example of a special defining function, that is, $|dr_2|^2_{r_2^2g_H}\equiv 1$ on a neighborhood of the boundary.  Moreover, $[g_H]_\infty = [(r_2^2g_H)|_{T^2\mathbb{S}^n}] = [(\text{round metric, $g_{round}$ on $\mathbb{S}^n$})]$, hence $(\mathbb{S}^n,[g_{round}])$ is the conformal infinity of $(\mathbb{H}^{n+1},g_H)$. Other examples can be found in the work of Chang--Qing--Yang \cite{ChangSun-YungA.2007SPiC} (see Section 2.2).

Recall that, for an AH manifold, we necessarily have $\left(|dr|^2_{\overline g}\right)|_\Sigma \equiv 1$ for any defining function $r$ of the boundary $\Sigma$. It is a result of Graham--Lee \cite{GrahamLee} that for any metric $\hat g\in [g_+]_\infty$ in the conformal infinity of an AH manifold $(M^{n+1},g_+)$, there is a special defining function $r$ for the boundary $\Sigma$ with $|dr|^2_{r^2g_+}\equiv 1$ in a collar neighborhood of $\Sigma$ and $(r^2g_+)|_{T\Sigma} = \hat g$. Such a defining function is uniquely determined in a neighborhood of the boundary. After making an identification, the metric can be expressed in normal form as \begin{equation}\label{NormalForm} \overline g = r^2 g_+ = dr^2 + g_r, \end{equation} where $g_r$ is a one-parameter family of metrics on $\Sigma$ with $g_0 = \hat g$. As we will see, the expansion of $g_r$ at $r = 0$ contains significant information about the Riemannian geometry of the space and the conformal geometry of the boundary, and it is the main focus of this work to explore this connection.

\subsection{Discussion of Main Results}

An important invariant of a complete and non-compact manifold $(M^{n+1},g_+)$ is the first $p$-Dirchlet eigenvalue ($1<p<\infty$). It is defined as
\begin{equation}\label{First-pDirichlet}
\lambda_{1,p}(M):=\inf_{\Omega} \lambda_{1,p}(\Omega) = \inf_{\Omega} \inf_{f\in W^{1,p}_o(\Omega)\setminus \{0\}} \frac{\displaystyle\int_\Omega |\nabla_{g_+}f|^p\;dv_{g_+}}{\displaystyle\int_\Omega |f|^p\;dv_{g_+}},
\end{equation}
where the infimum is taken over all bounded domains $\Omega$ in $M$ with smooth boundary. Notice that $\lambda_{1,p}(M)$ depends on $g_+$, so a more proper notation would be $\lambda_{1,p}(M,g_+)$. However, the metric $g_+$ will always be clear from the context. 

In \cite{pérezayala-tyrrell}, the authors extended a key consequence of one of Mazzeo's results in  \cite{MazzeoRafe1991UCaI} and showed that for any AH manifold $(M^{n+1},g_+)$,
\begin{equation}
\lambda_{1,p}(M)\le \left(\frac{n}{p}\right)^p.
\end{equation}
This upper bound is sharp, as it is achieved on $\mathbb{H}^{n+1}(-1)$. In particular, $\lambda_{1,2}(M) \leq (\frac{n}{2})^2$. In his remarkable work \cite{LeeJohnM.1995Tsoa}, Lee proved that if $(M^{n+1}, g_+)$ is a Poincar\'e--Einstein manifold whose conformal infinity has non-negative Yamabe type, then equality holds: $\lambda_{1,2}(M) = (\frac{n}{2})^2$. However, his result is not optimal in the sense that there exist Poincaré--Einstein manifolds with $\lambda_{1,2}(M) = (\frac{n}{2})^2$ where the conformal infinity has a negative Yamabe constant. This led Lee to pose the following two types of questions, both of which appear in the introduction of his work \cite{LeeJohnM.1995Tsoa}: 

\begin{enumerate}\label{LeeQuestions}
    \item Given a PE manifold $(M^{n+1},g_+)$, can we find necessary and sufficient conditions on the conformal infinity $(\Sigma, [g_+]_\infty)$ for $\lambda_{1,2}(M)=\frac{n^2}{4}$?
    \item Given a PE manifold $(M^{n+1},g_+)$, can we find necessary and sufficient conditions on $g_+$ for the conformal infinity to be of non-negative Yamabe type? 
\end{enumerate}

To our knowledge, to this day question (1) remains open. The situation is different for question (2). We would like to mention two results in this direction. First, Guillarmou--Qing \cite{GuillarmouColin2010SCoP}  answered question (2) by proving that, when $n+1>3$, if $g_+$ is Poincar\'e--Einstein, then the first real scattering pole is less than $\frac{n}{2}+1$ if and only if $(\Sigma,[g_+]_\infty)$ has positive Yamabe constant (see Theorem 1.1 in \cite{GuillarmouColin2010SCoP}). As they pointed out, their result still holds if $\Ric_{g_+}=- ng_+$ is replaced by the weaker assumption that $\Ric_{g_+}\ge -ng_+$ and the metric $g_r$ has expansion
\begin{equation}\label{WE-expansion}
g_r = g_{(0)} + g_{(1)}\cdot r+ g_{(2)}\cdot r^2 + O(r^3)= \hat g - P_{\hat g}\cdot r^2 + O(r^3),
\end{equation}
where $r$ is any special defining function, $\hat g = (r^2g_+)|_{T\Sigma}$ and 
\begin{equation}\label{SchoutenTensor}
P_{\hat g} = \frac{1}{n-2}\left(\Ric_{\hat g} - \frac{R_{\hat g}}{2(n-1)}\hat g\right)
\end{equation}
is the Schouten tensor with respect to $\hat g$; see Remark 1.3 in \cite{GuillarmouColin2010SCoP}. Asymptotically hyperbolic manifolds satisfying (\ref{WE-expansion}) are known as Weakly Poincar\'e--Einstein (WPE) manifolds. This is because, in the case of actual Poincar\'e--Einstein manifolds, the expansion of the metric $g_r$ is even, up to $n-1$ if $n$ is odd and up to $n$ if $n$ is even (Lemma 4 in \cite{ChangSun-YungA.2007SPiC}), and the first two terms are exactly as in (\ref{WE-expansion}) (see section 2 in \cite{Graham}).

The second result that we would like to remark addressing Lee's second question is that of Hijazi--Montiel--Raulot in \cite{HijaziOussama2020TCCo}. The three authors established that, when $n+1\ge 3$, if 
\begin{equation}\label{DecayCurvatureConditions}
\Ric_{g_+}+ng_+\ge 0\quad\text{ and }\quad R_{g_+}+n(n+1) = o(r^2),
\end{equation}
then the Cheeger constant $\Ch(M,g_+)$ (see (\ref{CheegerConstantDef}) for definition) is $n$ if and only if the Yamabe constant of the conformal infinity is non-negative; see Theorem 6 in \cite{HijaziOussama2020TCCo}. A natural question is how the results of Guillarmou--Qing and Hijazi--Montiel--Raulot relate to each other. This is the motivation behind the following result. 
\begin{theorem}\label{WPE+RiciffScDecayCond}
 Assume $n+1> 3$ and that $(M^{n+1},g_+)$ is an AH manifold satisfying $\Ric_{g_+} + ng_+\ge 0$. Then: \newline $g_{(1)}=0$ and $g_{(2)}=P_{\hat{g}}$, (i.e. $g_+$ is WPE) if and only if $R_{g_+}+n(n+1) = o(r^2).$ 
\end{theorem}
\noindent In summary, for this class of AH metrics $g_+$, having a first real scattering pole less than $\frac{n}{2}+1$ is equivalent to $(\Sigma,[g_+]_\infty)$ being of positive Yamabe type, and any of these $g_+$ have a Cheeger constant equal to $n$. When $n=2$, being WPE does not make sense since $P_{\hat{g}}$ is undefined and $g_{(2)}$ is not locally determined by the Einstein condition \cite{Graham}, but we can still consider the condition on the scalar curvature appearing in  (\ref{DecayCurvatureConditions}).

An important observation from the proof of Theorem \ref{WPE+RiciffScDecayCond} is that for an AH manifold with $\Ric_{g_+}+ng_+\ge 0$ and whose boundary is totally geodesic under compactification by a special defining function $r$, being WPE implies the vanishing of $\overline W_{0i0j}$ at the boundary. Here, $\overline W_{0i0j}$ denotes the components of the Weyl tensor with respect to $\overline g = r^2g_+$ in holographic coordinates - see the beginning of Section \ref{Preliminaries}. The quantity $\overline{W}_{0i0j}$ restricted to the boundary is called the third conformal fundamental form of $\Sigma$, with the second conformal fundamental form being the traceless second fundamental form of $\Sigma$. This terminology is borrowed  from the work of Blitz–Gover–Waldron \cite{BlitzSamuel2021CFFa}. In their work, Blitz–Gover–Waldron study a sequence of conformally invariant, higher-order fundamental forms to determine when a CC manifold is conformally equivalent to an asymptotically Poincaré--Einstein space (see Theorem 1.8 in \cite{BlitzSamuel2021CFFa}).\footnote{We must emphasize that the third conformal fundamental form is an extrinsic conformal hypersurface invariant, and its definition is independent of the choice of coordinate system; see \cite{BlitzSamuel2021CFFa}.} 

Recall that a manifold is said to be WPE if the complete metric $g_+$ exhibits the same asymptotic expansion as that of a PE metric, at least up to, and including, order 2; see (\ref{WE-expansion}). This is particularly relevant to our discussion, as we expect WPE manifolds, and their natural generalizations, to share many properties with PE manifolds. Consequently, understanding the conditions that lead to being WPE is crucial. As we will demonstrate, there exist large classes of AH manifolds that do not satisfy the condition $\Ric_{g_+}+ng_+\ge0$. This observation provides the motivation for the following result.

\begin{theorem}\label{WPE-Characterization}
     Let $(M^{n+1}, g_+)$ be an AH manifold, and let $\overline g = r^2g_+$ be a compactification of it by a special defining function $r$. The following holds:
    \begin{enumerate}
        \item $(n+1\ge 2)$ $g_{(1)} = 0$ if and only if $\Ric_{g_+}+ng_+ = O(1)$ if and only if $R_{g_+}+n(n+1) = o(r)$ and $\Sigma$ is umbilic in $(\overline{M}^{n+1}, \overline g)$.
        \item $(n+1>2)$ $g_{(1)} = 0$ and $g_{(2)} = -P_{\hat g}$, (i.e. $g_+$ is WPE) if and only if $\Ric_{g_+}+ng_+ = O(r)$ if and only if $R_{g_+}+n(n+1)=o(r^2)$, $\Sigma$ is umbilic and $\overline{W}_{0i0j}$ vanishes on $\Sigma$.
    \end{enumerate}
\end{theorem}

We finish the current subsection with a discussion of yet another type of AH manifold $(M^{n+1},g_+)$, which is more general than Poincar\'e--Einstein. Instead of requiring the complete metric to be PE, that is, to satisfy $\Ric_{g_+}+ng_+ = 0$, we could require the weaker condition $R_{g_+}+n(n+1) = 0$. Complete metrics satisfying $R_{g_+}+n(n+1)=0$ are known as singular Yamabe (SY) metrics, while we refer to metrics satisfying $R_{g_+}+n(n+1) = O(r^\alpha)$, for some $\alpha>2$, as asymptotically singular Yamabe (ASY) metrics. If we start with a CC manifold $(M^{n+1},g_+)$ and $\overline g$ is a smooth compactification of it, then there exists a unique defining function $u_{sy}$ for $\Sigma$ such that $g_{sy} := u_{sy}^{-2}\overline g$ is a singular Yamabe metric; see introduction of \cite{GrahamCRobin2019Cffs} and references therein. More generally, we could always find a defining function $u$ for $\Sigma$ such that $g_{asy}:=u^{-2}\overline g$ satisfies $R_{g_{asy}} +n(n+1)=O(r^{n+2}\log r)$, where $r$ is any special defining function \cite{GrahamC.Robin2017Vrfs}. Notice that both of these metrics satisfy the scalar curvature decay condition in (\ref{DecayCurvatureConditions}), and also the corresponding conformal infinities of $g_+$, $g_{sy}$ and $g_{asy}$ remain the same. As a consequence of Theorem \ref{WPE-Characterization}, we derive
\begin{corollary}\label{WPE-CharacterizationCor}
    Let $(M^{n+1}, g_+)$ be an asymptotically singular Yamabe metric, and let $\overline g$ be a compactification of it by a special defining function $r$. Then
    \begin{enumerate}
        \item $g_{(1)} = 0$ if and only if $\Ric_{g_{+}}+ng_{+} = O(1)$ if and only if $\Sigma$ is umbilic in $(\overline{M}^{n+1}, \overline g)$.
        \item $g_{(1)} = 0$ and $g_{(2)} = -P_{\hat g}$, (i.e. $g_{+}$ is WPE) if and only if $\Ric_{g_{+}}+ng_{+} = O(r)$ if and only if $\Sigma$ is umbilic and $\overline{W}_{0i0j}$ vanishes on $\Sigma$.
    \end{enumerate}
\end{corollary}

It immediately follows from Corollary \ref{WPE-CharacterizationCor}, and the preceding discussion, that a conformally compact manifold with umbilic boundary is conformally weakly Poincar\'e--Einstein if and only if the third conformal fundamental form of the boundary vanishes. Theorem \ref{WPE-Characterization}, and its Corollary \ref{WPE-CharacterizationCor}, are analogues of Theorem 1.8 in \cite{BlitzSamuel2021CFFa}, where the condition WPE is used in place of asymptotically Poincar\'e--Einstein. 

A quick calculation using (\ref{TransLawRiemann}) shows that any conformally compact manifold that satisfies $R_{g_+}+n(n+1) = o(r^2)$ is an asymptotically hyperbolic manifold. Therefore, inside the collection of Conformally Compact manifolds with $n+1>3$, 
\[
\{\text{PE}\}\subset \{R_{g_+}+n(n+1) = o(r^2)\text{ and }\Ric+ng\ge0\} = \{\text{WPE and }\Ric+ng\ge0\}\subset \{\text{AH}\},
\]
where the equality is thanks to Theorem \ref{WPE+RiciffScDecayCond}. Notice that SY metrics are a particular case of CC metrics $g_+$ with $R_{g_+} + n(n+1) = o(r^2)$, but they might not satisfy the Ricci condition. It is our interest to continue exploring Lee's questions \ref{LeeQuestions} in the more general context of AH manifolds that fail to satisfy the expansion in (\ref{WE-expansion}), the lower bound on the Ricci tensor, or both.

\subsubsection{Cheeger Constants of Asymptotically CMC Submanifolds}

In previous work, the authors introduced the notion of asymptotically CMC  submanifolds; see Definition 1.6 in \cite{pérezayala-tyrrell}. These are immersed submanifolds $\iota: Y^{k+1}\to (M^{n+1},g_+)$ which are conformally compact themselves, meet $\partial M$ transversely at  $\partial Y^k\subset \partial M = \Sigma^n.$ Furthermore the mean curvature $H$ satisfies
\begin{equation}\label{MeanCurvature-A-CMC}
g_+(H,H) = C^2+O(r)
\end{equation}
for a defining function $r$ - see Figure \ref{AH-Sub}. Since $r=O(\rho)$ for any pair $r$ and $\rho$ of defining functions for $\Sigma$, asymptotically CMC is well-defined. In this context, we refer to the ambient CC manifold $(M^{n+1},g_+)$ as the bulk manifold. Also, as mentioned, we will assume these submanifolds are $C^{m,\alpha}$ for $m\geq 3, \alpha\in (0,1).$

It turns out that if the bulk manifold $(M^{n+1},g_+)$ is AH, then the angle at which $ Y$ meets $\Sigma$ at infinity is constant and determined by $C$. Furthermore, $Y^{k+1}$ is asymptotically hyperbolic itself with asymptotic sectional curvatures equal to $(-1)(1-C^2(k+1)^{-1})$ - see Proposition 1.8 and its proof in \cite{pérezayala-tyrrell}. In particular, asymptotically minimal submanifolds (i.e. $C=0$) are AH themselves. These are generally not weakly Poincar\'e--Einstein and do not satisfy $\Ric_{h_+}+kh_+\ge0$. In fact, and as we will see, many  satisfy $\Ric_{h_+}+kh_+\le 0$. Therefore, the class of asymptotically minimal submanifolds provides a good source of interesting examples of AH spaces, especially when situated inside hyperbolic space, as they inherit some of the properties of the bulk manifold. 

\begin{figure}[h]
    \centering
\resizebox{0.72\textwidth}{!}{%
  \begingroup
\definecolor{diagramblue}{RGB}{8,120,245}

\begin{tikzpicture}[
    x=1cm,
    y=1cm,
    line cap=round,
    line join=round
]
\path[use as bounding box] (0,0) rectangle (10,10);

\draw[black,line width=2.4pt]
    (5,5) circle[radius=4.86];

\draw[diagramblue,line width=2.4pt]
    (5.25,9.80)
    .. controls (4.54,7.82) and (3.15,3.72) ..
    (4.42,2.83)
    .. controls (5.58,2.01) and (7.83,4.21) ..
    (8.91,4.49)
    .. controls (9.25,4.59) and (9.57,4.55) ..
    (9.80,4.36);

\draw[diagramblue,line width=2.4pt]
    (5.25,9.80)
    .. controls (5.92,8.02) and (6.48,6.78) ..
    (7.64,5.80)
    .. controls (8.25,5.29) and (9.13,4.70) ..
    (9.80,4.36);

\draw[
    diagramblue,
    line width=2.4pt,
    dash pattern=on 7pt off 7pt
]
    (5.28,9.78)
    .. controls (7.10,9.12) and (8.95,7.16) ..
    (9.80,4.36);

\draw[diagramblue,line width=2.4pt]
    (4.68,3.58)
    .. controls (5.00,3.55) and (5.25,3.55) ..
    (5.50,3.67)
    .. controls (6.15,3.78) and (6.84,4.54) ..
    (7.10,4.92)
    .. controls (7.24,5.11) and (7.33,5.32) ..
    (7.40,5.49);

\draw[diagramblue,line width=2.4pt]
    (5.50,3.67)
    .. controls (5.88,4.20) and (6.55,4.75) ..
    (7.10,4.92);

\node[font=\Large] at (2.82,7.91) {$\Sigma^{n}=\partial M$};
\node[font=\Large] at (1.42,3.13) {$M^{n+1}$};
\node[font=\Large] at (4.74,4.48) {$Y^{k+1}$};
\node[font=\Large] at (7.14,7.48) {$\partial Y$};

\end{tikzpicture}
\endgroup%
}
    \caption{A conformally compact submanifold $Y^{k+1}$ in an AH space $M^{n+1}$.}
    \label{AH-Sub}
\end{figure}
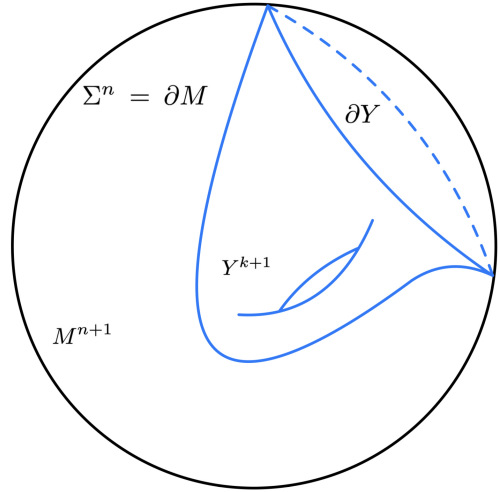

The Cheeger constant $\Ch(M, g_+)$ is defined as the isoperimetric ratio
\begin{equation}\label{CheegerConstantDef}
\Ch(M, g_+) = \inf_{\Omega} \frac{A(\partial \Omega)}{V(\Omega)},
\end{equation}
where the infimum is taken over all compact smooth domains in $M^{n+1}$, and both geometric quantities are computed with respect to the induced metrics coming from $g_+$. When referring to the Cheeger constant of an asymptotically CMC submanifold $Y^{k+1}$, the infimum is taken over compact smooth domains within $Y^{k+1}$, and all quantities are computed with respect to the induced metric $h_+$. We are now in position to state the following:
\begin{theorem}\label{CHCONSTANT-UpperBound}
    Let $Y^{k+1}$ be a conformally compact submanifold inside an AH space $(M^{n+1},g_+)$ which has asymptotically constant mean curvature $C$. Then 
    \[
    \Ch(Y^{k+1})\le k\left(1-\frac{C^2}{(k+1)^2}\right)^{\frac{1}{2}}.
    \]
    In particular, if $C=0$ (i.e. $Y$ is asymptotically minimal), then 
    \[
    \Ch(Y^{k+1})\le k.
    \]
\end{theorem}

Before we introduce our next result, the following definition is needed: 
\begin{definition}
    Let $(M^{n+1},g_+)$ be an AH space. We call a smooth, positive solution $u$ of 
    \begin{equation}\label{LeeEq}
    \begin{cases}
        \Delta_{g_+}u = (n+1)u \quad \text{ on }M^{n+1}\\ u - r^{-1} = O(1) \quad \text{ as }r\to0^+
    \end{cases}
    \end{equation}
    a Lee-eigenfunction. 
\end{definition}
Of particular interest are the Lee-eigenfunctions which also satisfy $|\nabla_{g_+}u|^2_{g_+} \le u^2$ on $M^{n+1}.$ 
\noindent On an AH manifold, given any smooth defining function $r$, there exists a unique smooth and positive solution $u$ to (\ref{LeeEq}) -- see Proposition 4.1 in \cite{LeeJohnM.1995Tsoa}. This solution may not satisfy the gradient estimate. However, if the manifold is PE and the conformal infinity is of nonnegative Yamabe type, then there exists a Lee-eigenfunction $u$ for which the gradient estimate holds (Theorem A in \cite{LeeJohnM.1995Tsoa}). As pointed out by Guillarmou--Qing in \cite{GuillarmouColin2010SCoP} (Remark 1.3), a Lee-eigenfunction satisfying the gradient estimate still exists if PE is replaced by $\Ric_{g_+}+ng_+\ge 0$ and $(M^{n+1},g_+)$ is WPE.

Theorem \ref{CHCONSTANT-UpperBound} reflects the fact that knowing the behavior of the AH manifold at infinity is sufficient to get upper bounds on Cheeger constants. Lower bounds, however, require global information.

\begin{theorem}\label{LowerBoundCheegerConstant-Submanifold}
Let $Y^{k+1}\to M^{n+1}$ be a complete, non-compact immersion into a WPE manifold $(M^{n+1},g_+)$ with $\Ric_{g_+}+ng_+\ge 0$ and whose conformal infinity $(\Sigma,[g_+]_\infty)$ has non-negative Yamabe invariant, and let $u$ be a Lee-eigenfunction satisfying $\frac{|\nabla u|^2}{u^2}\leq 1$. Denote by $b(u) = \nabla^2_{g_+}u - ug_+$ the trace-free hessian of $u$. If the mean curvature vector of $Y$ has norm satisfying $\|H^Y\|\le \alpha$ for some constant $\alpha$, and if 
\[
\beta^Y(u):=\underset{Y}{\sup} \left(u^{-1}\cdot tr\left(b(u)|_{(TY^{\perp}){}^2}\right)\right)
\] 
satisfies $\beta^Y(u)+\alpha<k$, then it follows that
\[
\Ch(Y^{k+1})\ge k-\beta^Y(u)-\alpha>0.
\]
Hence,
\[
\Ch(Y^{k+1})\ge k-\hat \beta^Y-\alpha>0, 
\]
where $\hat \beta^Y$ is defined as the infimum of $\beta^Y(u)$ over all Lee-eigenfunctions satisfying the gradient estimate. 
\end{theorem}

We would like to remark that $\hat \beta^Y$ is an invariant of $h_+$ that naturally arises in the study of eigenvalue estimates on submanifolds; see Theorem 1.12 in \cite{pérezayala-tyrrell}. If $Y^{k+1}$ is a minimal and conformally compact submanifold of a PE manifold with conformal infinity of non-negative Yamabe type, then $\hat \beta^Y>0$ if $Y^{k+1}$ possesses $L^2$-eigenvalues (Corollary 1.13 in \cite{pérezayala-tyrrell}), and it would be interesting to study whether the converse holds. If the bulk manifold is $\mathbb{H}^{n+1}(-1)$, then it turns out that $\hat \beta^Y = 0$ (see the proof of Corollary 1.14 in \cite{pérezayala-tyrrell}) and we obtain the following corollary:
\begin{corollary}\label{CheegerConstantSubmanifoldInH}
    Let the same assumptions be as in Theorem \ref{LowerBoundCheegerConstant-Submanifold}. If we further assume that $(M^{n+1},g_+) = (\mathbb{H}^{n+1}(-1),g_H)$, then $\hat 
    \beta^Y = 0$ for any such submanifold. Therefore, 
    \[
    \Ch(Y^{k+1})\ge k-\alpha>0.
    \]
    Furthermore, if such $Y$ is minimal, then 
    \[
    \Ch(Y^{k+1})\ge k.
    \]
\end{corollary}

Finally, combining Theorem \ref{CHCONSTANT-UpperBound} and Theorem \ref{LowerBoundCheegerConstant-Submanifold}, we deduce
\begin{theorem}\label{CheegerConstant-SubHyperbolicSpaceI}
    Let the assumptions be as in Theorem \ref{LowerBoundCheegerConstant-Submanifold}. If we further assume that $Y$ is a conformally compact submanifold which is  minimal, then 
    \[
    k\ge \Ch(Y^{k+1})\ge k-\hat \beta^Y.
    \]
    Furthermore, if the bulk space is $\mathbb{H}^{n+1}(-1)$, then $\Ch(Y^{k+1}) = k$.
\end{theorem}

The results of this section provide us with tools for constructing examples where some of the conclusions of the results due to Guillarmou--Qing and Hijazi--Montiel--Raulot still hold in a larger class of AH manifolds. These examples are discussed in the following subsection.

\subsubsection{Examples of AH manifolds with \texorpdfstring{$\Ch(M^{n+1},g_+) = n$}\;  and \texorpdfstring{$\mathcal{Y}(\Sigma,[g_+]_\infty)$}\; of any sign.}

We consider minimal and conformally compact submanifolds $Y^{n+1}$ within $\mathbb{H}^{n+2}(-1)$. First, thanks to Proposition 1.8 in \cite{pérezayala-tyrrell}, $Y^{n+1}$ is asymptotically hyperbolic itself. Second, it follows from Theorem \ref{CheegerConstant-SubHyperbolicSpaceI} that $\Ch(Y^{n+1}) = n$. Third, for these submanifolds, it holds that $\Ric^Y + nh_+ = -B^2\le 0$, where $B$ is the second fundamental form of $Y^{n+1}$ (see (\ref{RicciSubmanifod2}) in the appendix), and equality cannot hold everywhere if, for instance, their conformal infinity is of negative Yamabe type. This is because minimal submanifolds inside hyperbolic space are totally geodesic only if their boundary is a lower dimensional copy of the standard round sphere \cite{Spivak}. Moreover, conformally compact minimal submanifolds $Y^{n+1}$ of $\mathbb{H}^{n+2}(-1)$ can have conformal infinities of different Yamabe types. In particular, it could be of negative Yamabe type, at least in the $n+1=3$ case, showing that $\Ch(Y^{n+1}) = n$ alone cannot characterize non-negative boundary Yamabe constant in the space of non-Einstein, AH manifolds. 

The two classes of examples we construct are discussed in the following theorem:
\begin{theorem}\label{EX-PosNegYamabe}
There are examples of $(n+1)$-dimensional asymptotically hyperbolic spaces \newline$(Y^{n+1},h_+)$ such that
\begin{enumerate}
    \item $\Ch(Y^{n+1}) = n$ and whose conformal infinity is of positive Yamabe type ($n+1\ge 3$);
    \item $\Ch(Y^{n+1}) = n$ and whose conformal infinity is of negative Yamabe type ($n+1=3$). 
\end{enumerate}
The existence of these examples, which do not satisfy $\Ric_{h_+}+nh_+\ge 0$, implies that the conditions in (\ref{DecayCurvatureConditions}) may not be omitted from Theorem 6 in \cite{HijaziOussama2020TCCo}. Finally, some of the AH manifolds satisfying (1) have a disconnected boundary.
\end{theorem}

We now examine whether these examples satisfy the WPE condition. Note that since they do not fulfill the Ricci lower bound, the characterization provided by Theorem \ref{WPE+RiciffScDecayCond} is no longer applicable. Hence, a different approach is needed. Our result on this issue pertains to asymptotically minimal hypersurfaces with mean curvature $H^Y=O(r^2)$, which motivates the following proposition of independent interest:
\begin{proposition}\label{conf inv}
Let $(M^{n+2},g_+)$ be Poincar\'e--Einstein manifold manifold. Let $Y^{n+1}$ be an asymptotically minimal hypersurface of $M$, then $\hat{\eta}-n\overline{B}_{00}$ is a conformal invariant of $\partial Y$ of weight $-1$ (see (\ref{NewConfInv})) which vanishes if and only if $H^Y=O(r^2)$. Here $\hat \eta$ is the mean curvature of $\partial Y$ in $\Sigma$, while $\overline B_{00}$ denotes the second fundamental form of $\overline Y$ in $\overline M$ evaluated on $(\partial_r,\partial_r)$ at points of $\partial Y.$ 
\end{proposition}

Finally, our next result provides a characterization for WPE, asymptotically minimal hypersurfaces living inside a PE manifold. 

\begin{theorem}\label{WPE-Submanifolds}
    Assume $n+1>3$. Let $(M^{n+2},g_+)$ be a PE space. Let $Y^{n+1}$ be an asymptotically minimal hypersurface with mean curvature satisfying $H^Y=O(r^2).$ Then $(Y^{n+1}, \hat h)$ is WPE if and only if its boundary $\partial Y$ is umbilic in $(\Sigma, \hat g)$ and its third conformal fundamental form as a hypersurface in $\Sigma$ vanishes.
\end{theorem}

We would like to make some clarifications regarding  Theorem \ref{WPE-Submanifolds}. Recall that any asymptotically minimal submanifold within an AH space is AH itself; Proposition 1.8 in \cite{pérezayala-tyrrell}. Therefore, for $\hat h$ as in the statement of the theorem, it follows from Graham--Lee \cite{GrahamLee} that there is a special defining function $\rho$ for $\partial Y$ such that $\tilde h := \rho^2h_+$ restricts to $\hat h$ on $\partial Y$. The new compactified metric $\tilde h$ can now be written in normal form near the boundary and expanded in terms of $\rho$ at $\rho=0$, 
\[
\tilde h = d\rho^2+h_\rho = d\rho^2 + \hat h + h_{(1)}\rho + h_{(2)}\rho^2 + O(\rho^3).
\]
Being WPE then means that \( h_{(1)} = 0 \) and \( h_{(2)} = -P_{\hat h} \). Notably, \( h_{(1)} = 0 \) is equivalent to \( \partial Y \) being umbilic within \( (\overline{Y}, \tilde{h}) \). What is significant in Theorem \ref{WPE-Submanifolds} is that the WPE condition is characterized using extrinsic information.

As a consequence of Theorem \ref{WPE-Submanifolds}, we argue that the spherical catenoids in hyperbolic space satisfy part (1) of Theorem \ref{EX-PosNegYamabe}, and are WPE. However, these examples have a disconnected boundary, demonstrating that the result of Witten–Yau \cite{WittenYau}, which asserts the connectedness of the conformal infinity of a PE manifold with a conformal infinity of positive Yamabe type, does not extend to the broader class of WPE manifolds. In fact, Hijazi–Montiel–Raulot showed in \cite{HijaziOussama2020TCCo} (Theorem 5) that Witten–Yau’s result holds on WPE manifolds that satisfy the Ricci lower bound (\ref{DecayCurvatureConditions}), not just on PE manifolds. Our examples show that the Ricci lower bound cannot be dropped. On the other hand, the examples given in Theorem \ref{EX-PosNegYamabe}, part (2), do not satisfy the WPE condition and therefore fail to meet the decay condition on the scalar curvature appearing in (\ref{DecayCurvatureConditions}). This implies that the scalar curvature decay condition cannot be omitted from the assumptions in Corollary 4 of \cite{HijaziOussama2020TCCo}.

\subsection{Organization of the Paper.} The paper is organized as follows. In Section \ref{Preliminaries}, we provide a self-contained and necessary background on asymptotically hyperbolic (AH) manifolds. In particular, we give full details on the expansion of the Riemannian metric up to third order, and of the mean curvature and second fundamental form up to second order of the $r$-level set, all in the general case of an AH manifold in holographic coordinates. In Section \ref{WPE+Ric-Manifolds}, we present the proofs of Theorem \ref{WPE+RiciffScDecayCond} and Theorem \ref{WPE-Characterization}. Of particular interest here are the expansions of both the Ricci tensor and the scalar curvature in holographic coordinates for an AH metric. In Section \ref{CheegerConstantOnSubmanifolds}, we prove Theorem \ref{CHCONSTANT-UpperBound}, Theorem \ref{LowerBoundCheegerConstant-Submanifold}, and Theorem \ref{CheegerConstant-SubHyperbolicSpaceI}, together with Corollary \ref{CheegerConstantSubmanifoldInH}. As in our previous work \cite{pérezayala-tyrrell}, we continue to use Lee-eigenfunctions to relate the geometry of the bulk manifold to that of the asymptotically minimal submanifolds, a technique that could be useful and of interest when addressing questions about this type of submanifold. Finally, in Section \ref{Examples}, we provide the proof of Theorem \ref{EX-PosNegYamabe} and Theorem \ref{WPE-Submanifolds}. A key consequence of our work is the construction of a large class of AH manifolds that are neither WPE nor satisfy the Ricci lower bound (\ref{DecayCurvatureConditions}), yet still attain the sharp value for both the Cheeger constant and the first Dirichlet eigenvalue. We believe these manifolds could provide many other interesting examples and help shed light on open questions in the field. Finally, in Section \ref{Examples}, we also prove Proposition \ref{conf inv}, where a new extrinsic conformal invariant arises as the obstruction to the vanishing of the mean curvature of an asymptotically minimal hypersurface to second order. This invariant is coupled in the sense that it depends on the extrinsic geometry of the boundary of the hypersurface within both the hypersurface itself and within the boundary of the bulk manifold.

\section{Preliminaries}\label{Preliminaries}
Let $(M^{n+1},g_+)$ be an AH manifold and let $\overline g = r^2g_+$ be a compactification of it, where $r$ is a special defining function. The flow of $\nabla_{r^2g_+}r$ determines a diffeomorphism $\phi \colon C\to \Sigma\times [0,\epsilon_0)_r,$ where $C$ is a neighbourhood of $\Sigma$ in $M.$ $\phi$ is the inverse of the map $[(p,r)\mapsto F(p,r)]$ where $[r\mapsto F(p,r)]$ is the integral curve of $\nabla_{r^2g_+}r$ emanating from $p.$    This provides us with an identification of this neighborhood of $\Sigma$ in $M$ with $\Sigma\times [0,\epsilon_o)_r$ on which the metric can be written as:
\begin{equation}
    \overline g = dr^2 + g_r,
\end{equation}
where $g_r$ is a one-parameter family of metrics on $\Sigma$ with $g_0 = \hat g = (r^2g_+)|_{T\Sigma}$ being the induced metric on the boundary. With this identification we may extend any local coordinates on a neighborhood within $\Sigma$ to local coordinates on a neighborhood within $M,$ we call such an extended coordinate system a holographic coordinate system. Note that these are Fermi coordinates adapted to the boundary, but they are a useful medium through which to relate Riemannian geometry of the interior to conformal geometry of the ideal boundary.

\begin{figure}[h]
    \centering
   \resizebox{\textwidth}{!}{%
  \begingroup
\definecolor{diagramblue}{RGB}{0,118,245}

\begin{tikzpicture}[
    x=1cm,
    y=1cm,
    line cap=round,
    line join=round,
    >={Stealth[length=3.7mm,width=2.7mm]},
    every node/.style={font=\Large}
]
\path[use as bounding box] (0,0) rectangle (15,6);

\draw[line width=1.15pt]
    (2.13,5.90)
    .. controls (3.35,5.10) and (4.36,4.20) .. (4.82,3.30)
    .. controls (5.25,2.47) and (5.42,1.35) .. (5.24,0.50);

\draw[
    diagramblue,
    line width=1.45pt,
    dash pattern=on 8pt off 7pt
]
    (0.10,5.10)
    .. controls (1.85,4.02) and (3.07,2.65) .. (3.25,0.60);

\node at (0.72,2.74) {$M^{n+1}$};
\node at (3.27,5.55) {$\Sigma^n$};

\draw[
    diagramblue,
    line width=1.45pt,
    -{Stealth[length=3.6mm,width=2.7mm]}
]
    (4.55835,3.73159) -- (3.36990,2.90088);

\draw[diagramblue,line width=1.45pt]
    (4.71876,3.50210)
    -- (4.48927,3.34169)
    -- (4.32886,3.57118);

\node[anchor=north west] at (3.28,2.67)
    {$\nabla_{r^2g_+}r$};

\draw[
    line width=1.15pt,
    -{Stealth[length=3.7mm,width=2.7mm]}
]
    (6.10,3.00) -- (8.50,3.00);
\node at (7.22,3.67) {$\phi$};

\draw[
    line width=1.15pt,
    {Stealth[length=3.6mm,width=2.7mm]}-{Stealth[length=3.6mm,width=2.7mm]}
]
    (8.82,0.67) -- (14.24,0.67);

\draw[
    diagramblue,
    line width=1.45pt,
    -{Stealth[length=3.6mm,width=2.7mm]}
]
    (11.56,0.67) -- (11.56,4.72);

\draw[
    diagramblue,
    line width=1.45pt,
    dash pattern=on 9pt off 7pt,
    {Stealth[length=3.5mm,width=2.6mm]}-{Stealth[length=3.5mm,width=2.6mm]}
]
    (8.82,3.00) -- (14.30,3.00);

\node[anchor=south] at (11.56,4.92)
    {$\Sigma\times[0,\epsilon_o)_r$};
\node[anchor=west] at (11.68,4.56) {$r$};
\node[anchor=south west] at (11.61,3.03) {$\epsilon_o$};
\node[anchor=west] at (14.28,0.64) {$\Sigma^n$};

\end{tikzpicture}
\endgroup%
}
    \caption{Holographic coordinates near the ideal boundary of an AH manifold.}
    \label{HoloCoord}
\end{figure}
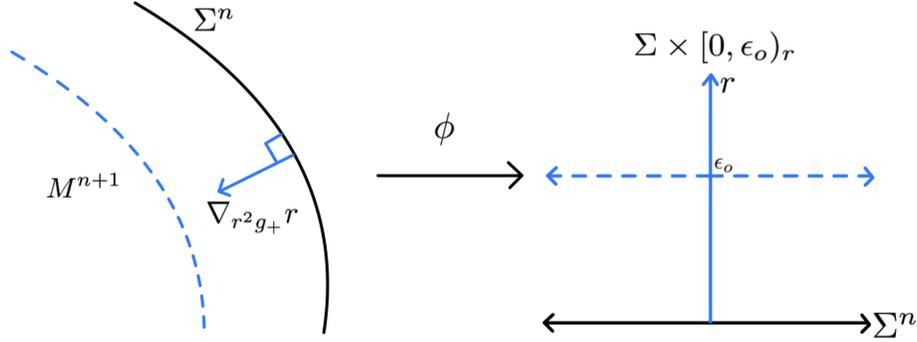

For $r\in[0,\epsilon_o)$, the $r$-level set $\Sigma_r = \{x\in M: r(x) = r\}$ of our special defining function is a smooth hypersurface. We denote by $ \overline L_r$  and $\overline{H}_r = \Tr_{\overline g}(\overline L_r)$ the second fundamental form and mean curvature, respectively, of  $\Sigma_r$ with respect to the compactified metric. More explicitly, $\overline H_r = \Tr_{\overline g}(\overline L_r) = \overline g^{ij}(\overline L_r)_{ij} = g_r^{ij}(\overline L_r)_{ij}$. All computations that follow will be done in this collar neighborhood and with respect to these coordinates. Latin letters are used for tangential directions, $r$ is used to denote the $\partial_r$ direction and $0$ is used to denote the direction of $\partial_r|_{r=0}$; we also use Greek letters to denote either. 

Our first goal is to compute the expansion of the metric $g_r$ up to the third order, see Proposition \ref{ExpansionMetricOrder3}. This has been computed in different places, for instance, some terms are computed in \cite{Graham} in the context of PE spaces and also in \cite{GrahamCRobin2019Cffs} in the context of singular Yamabe metrics. Here, we are working in the more general context of AH metrics and thus for the reader's convenience we include all details.

First, notice that  
\begin{equation}\label{FirstDerivative}
    (\overline L_r)_{ij} = \overline g(\overline  \Gamma_{ij}^\alpha\partial_\alpha, \partial_r) = \overline \Gamma_{ij}^r = -\frac{1}{2} g'_{ij}.
    \end{equation}
    This gives the first order term in the expansion of $g_r$ at $r=0$, that is,
    \[
    g_r = \hat g-2\hat L_{ij}\cdot r +\cdots,
    \]
    where $\hat L$ denotes the second fundamental form of $(\Sigma, \hat g)$. The other Christoffel symbols with respect to the compactified metric are given by
   \begin{equation}
    \overline \Gamma_{rj}^k = \frac{1}{2}\overline g^{ki}g'_{ij} = -(\overline L_r)^k_j.
    \end{equation}
    and
    \begin{equation}
    \overline \Gamma^r_{rj} = \overline \Gamma^r_{rr} = \overline \Gamma_{rr}^k = 0.
    \end{equation}
    The following proposition is discussed in \cite{GrahamCRobin2019Cffs} (see equation (2.12)), though most of the computations are omitted. We provide the details here for the reader's benefit. 

\begin{proposition}\label{ExpansionMetricOrder3}
Let $r$ be a special defining function and let $\overline g$ be the associated compactified metric. The expansion of the metric $(g_r)_{ij}$ near the boundary ($r=0$) is given by 
\begin{equation}\label{ExpansionMetric}
\begin{split}
g_{ij} =&\; g_{(0)} + g_{(1)}\cdot r + g_{(2)}\cdot r^2 + g_{(3)}\cdot r^3 + O(r^4)\\=&\; \hat g_{ij} + (-2\hat L_{ij})\cdot r + (-\overline R_{0i0j} + (\hat L)^a_i(\hat L)_{aj})\cdot r^2 \\&+ \frac{1}{3}(-\overline R_{0i0j,0} + ((\hat L)^k_i\overline R_{0k0j} + (\hat L)^k_j\overline R_{0k0i}))\cdot r^3+ O_{ij}(r^4),
\end{split}
\end{equation}
where $O_{ij}(r^4)$ denotes the components of a symmetric two-tensor.
\end{proposition}

\begin{proof}
We have already computed the first derivative of the metric. In order to compute the second and third derivative, we proceed as follows. The components of the Riemann tensor $\overline R_{rirj}$ with respect to the compactified metric $\overline g$ are given by 
\[
\begin{split}
    \overline R_{rirj} &= \overline R_{rir}^\alpha\overline g_{\alpha j} = \overline R_{rir}^kg_{kj} \\&=  \overline \Gamma_{rr,i}^kg_{kj} - \overline \Gamma_{ir,r}^kg_{kj} + (\overline \Gamma_{rr}^a \overline \Gamma_{ia}^k  - \overline \Gamma_{ir}^a \overline \Gamma_{ra}^k)g_{kj} \\& =  - \overline \Gamma_{ir,r}^kg_{kj} - \overline \Gamma_{ir}^a \overline \Gamma_{ra}^kg_{kj},
    \end{split}
    \]
where $g = g_r$. Moreover, its first derivative is
    \[
    \overline R_{rirj,r} = -\overline \Gamma^k_{ir,rr}g_{kj} - \overline \Gamma_{ir,r}^kg'_{kj} - \overline \Gamma_{ir,r}^a\overline \Gamma^k_{ra}g_{kj} - \overline \Gamma_{ir}^a\overline\Gamma_{ra,r}^kg_{kj} - \overline\Gamma_{ir}^a\overline\Gamma_{ra}^kg'_{kj}.
    \]

Let us start with the second derivative, which we claim to be given by 
\begin{equation}\label{SecondDerivative}
    g''_{ij} = -2\overline R_{rirj} + 2(\overline L_r)^a_i(\overline L_r)_{aj}.
    \end{equation}
We prove it by working out the terms appearing in $\overline R_{rirj}$. First, 
    \[
    \begin{split}
    \overline \Gamma_{ir,r}^kg_{kj} &= \frac{1}{2}\partial_r(g^{ka}g'_{ai})g_{kj} = \frac{1}{2}(g^{ka})'g'_{ai}g_{kj} +\frac{1}{2}g^{ka}g''_{ai}g_{kj} \\ &= -\frac{1}{2}g^{ka}g'_{ai}g'_{kj} +\frac{1}{2}g''_{ij} = -\frac{1}{2}g^{ka}(-2(\overline L_r)_{ai})(-2(\overline L_r)_{kj}) +\frac{1}{2}g''_{ij} \\&= -2(\overline L_r)^k_i(\overline L_r)_{kj} + \frac{1}{2}g''_{ij}.
    \end{split}
    \]
To compute the the second term in $\overline R_{rirj}$, we proceed as follows:
    \[
    \overline \Gamma_{ir}^a \overline \Gamma_{ra}^kg_{kj} = (\overline L_r)^a_i(\overline L_r)^k_ag_{kj} = g^{ab}g^{kc}(\overline L_r)_{ib}(\overline L_r)_{ac}g_{kj} = (\overline L_r)^a_i(\overline L_r)_{aj}.
    \]
Therefore,
    \[
    \begin{split}
    2\overline R_{rirj} &=  - 2\overline \Gamma_{ir,r}^kg_{kj} - 2\overline \Gamma_{ir}^a \overline \Gamma_{ra}^kg_{kj} \\&= 4(\overline L_r)^k_i(\overline L_r)_{kj} - g''_{ij} -2(\overline L_r)_i^a(\overline L_r)_{aj},
    \end{split}
    \]
thus
    \[
    g''_{ij} = -2\overline R_{rirj} + 2(\overline L_r)^a_i(\overline L_r)_{aj},
    \]
as claimed.

The third derivative of the metric is given by 
    \begin{equation}
        g'''_{ij} = -2\overline R_{rirj,r} + 2((\overline L_r)^k_i\overline R_{rkrj} + (\overline L_r)^k_j\overline R_{rkri})
    \end{equation}
and can be obtained from $\overline R_{rirj,r}$. First, 
    \[
        \overline \Gamma_{ir,rr}^k = \frac{1}{2}\partial^2_r(g^{ka}g'_{ai}) = \frac{1}{2}\partial_r((g^{ka})'g'_{ai}+ g^{ka}g''_{ai})= \frac{1}{2}((g^{ka})''g'_{ai}+ 2(g^{ka})'g''_{ai} + g^{ka}g'''_{ai})
    \]
Then $\overline \Gamma^k_{ir,rr}g_{kj}$ is given by
\[
\begin{split}
 &\frac{1}{2}(g^{ka})''g'_{ai}g_{kj} + (g^{ka})'g''_{ai}g_{kj}+ \frac{1}{2}g^{ka}g'''_{ai}g_{kj}\\&= (-(\overline L_r)_{ai})(-2(g^{ka})'g'_{kj} -g^{ka}g''_{kj})-(-2\overline R_{rari} + 2(\overline L_r)^b_a(\overline L_r)_{bi})g^{ka}(-2(\overline L_r){kj}) + \frac{1}{2}g'''_{ij}\\&= 2(\overline L_r)_{ai}(-2(\overline L_r)_{kj})(2(\overline L_r)^{ka})+(\overline L_r)^k_i(-2\overline R_{rkrj}+2(\overline L_r)^b_k(\overline L_r)_{bj})\\&\;\;\;\;+2(\overline L_r)^a_j(-2\overline R_{rari}+2(\overline L_r)^b_a(\overline L_r)_{bi}) + \frac{1}{2}g'''_{ij} \\ &= -8(\overline L_r)_{ai}(\overline L_r)_{kj}(\overline L_r)^{ka}+(\overline L_r)^k_i(-2\overline R_{rkrj}+2(\overline L_r)^b_k(\overline L_r)_{bj})\\&\;\;\;\;+2(\overline L_r)^a_j(-2\overline R_{rari}+2(\overline L_r)^b_a(\overline L_r)_{bi})+ \frac{1}{2}g'''_{ij}
\end{split}
\]
If we define $\overline L_r^3$ to be the symmetric two-tensor whose components are 
\[
(\overline L_r^3)_{ij} = (\overline L_r)_{ai}(\overline L_r)_{kj}(\overline L_r)^{ka}=(\overline L_r)^f_i(\overline L_r)^e_j(\overline L_r)_{ef} = (\overline L_r)^e_j(\overline L_r)^a_e(\overline L_r)_{ai},
\]
then we can write
\begin{equation}\label{3rdDerivative-Term1}
\overline \Gamma^k_{ir,rr}g_{kj} = \frac{1}{2}g'''_{ij}-2(\overline L_r)^3_{ij}-2((\overline L_r)^k_i\overline R_{rkrj}+2(\overline L_r)^a_j\overline R_{rari}).
\end{equation}
Next, we have
\begin{equation}\label{3rdDerivative-Term2}
\begin{split}
\overline \Gamma^k_{ir,r}g'_{kj}& = \frac{1}{2}((g^{ka})'g'_{ai}+ g^{ka}g''_{ai})(-2(\overline L_r)_{kj}) \\ &= - (\overline L_r)_{kj}(-4(\overline L_r)^{ka}(\overline L_r)_{ai} + g^{ka}(-2\overline R_{rari}+2(\overline L_r)^b_i(\overline L_r)_{ba})) \\&= 4(\overline L_r)^3_{ij} + 2(\overline L_r)^a_j\overline R_{rari} - 2(\overline L_r)^3_{ij} = 2(\overline L_r)^3_{ij}+ 2(\overline L_r)^a_j\overline R_{rari},
\end{split}
\end{equation}
and
\begin{equation}\label{3rdDerivative-Term3}
\begin{split}
\overline \Gamma_{ir,r}^a\overline \Gamma^k_{ra}g_{kj} &= \frac{1}{2}((g^{ab})'g'_{bi}+ g^{ab}g''_{bi}) (-(\overline L_r)^k_a)g_{kj}\\&=-\frac{1}{2}(\overline L_r)_{aj}(-4(\overline L_r)^{ab}L_{bi} + g^{ab}(-2\overline R_{rbri}+ 2(\overline L_r)^c_b(\overline L_r)_{ci})) \\&= 2(\overline L_r)^3_{ij} + (\overline L_r)^b_j\overline R_{rbri} - (\overline L_r)^3_{ij} = (\overline L_r)^3_{ij} + (\overline L_r)^b_j\overline R_{rbri}.
\end{split}
\end{equation}
We also need
\begin{equation}\label{3rdDerivative-Term4}
\begin{split}
\overline \Gamma_{ir}^a\overline\Gamma_{ra,r}^kg_{kj}&=-(\overline L_r)^a_i(-2(\overline L_r)^k_a(\overline L_r)_{kj} + \frac{1}{2}g''_{aj})= 2(\overline L_r)^3_{ij} - \frac{1}{2}(\overline L_r)^a_i(-2\overline R_{rarj} + 2(\overline L_r)^c_a(\overline L_r)_{cj}) \\&= (\overline L_r)^3_{ij} + (\overline L_r)^a_i\overline R_{rarj},
\end{split}
\end{equation}
and
\begin{equation}\label{3rdDerivative-Term5}
\overline\Gamma_{ir}^a\overline\Gamma_{ra}^kg'_{kj} = (-(\overline L_r)^a_i)(-(\overline L_r)^k_a)(-2(\overline L_r)_{kj}) = -2(\overline L_r)^3_{ij}.
\end{equation}
Putting together (\ref{3rdDerivative-Term1}), (\ref{3rdDerivative-Term2}), (\ref{3rdDerivative-Term3}), (\ref{3rdDerivative-Term4}) and (\ref{3rdDerivative-Term5}) gives
\[
\begin{split}
\overline R_{rirj,r} =& -\frac{1}{2}g'''_{ij}+2(\overline L_r)^3_{ij}+2((\overline L_r)^k_i\overline R_{rkrj}+2(\overline L_r)^a_j\overline R_{rari}) -2(\overline L_r)^3_{ij}\\&- 2(\overline L_r)^a_j\overline R_{rari} -(\overline L_r)^3_{ij} - (\overline L_r)^b_j\overline R_{rbri} - (\overline L_r)^3_{ij} - (\overline L_r)^a_i\overline R_{rarj}+2(\overline L_r)^3_{ij} \\=& -\frac{1}{2}g'''_{ij} +((\overline L_r)^k_i\overline R_{rkrj} + (\overline L_r)^k_j\overline R_{rkri})
\end{split}
\]
That is, 
\[
g'''_{ij} = -2\overline R_{rirj,r} + 2((\overline L_r)^k_i\overline R_{rkrj} + (\overline L_r)^k_j\overline R_{rkri}),
\]
as desired. This finishes the proof.
\end{proof}

\begin{remark}
The derivatives of the inverse components of the metric $g=g_r$ can be computed as follows. First, from 
\[
g^{ib}g_{bc} = \delta^i_c \implies (g^{ib})'g_{bc} + g^{ib}g'_{bc}=0 \implies (g^{ij})' = -g'_{bc}g^{ib}g^{cj}.
\]
Therefore,
\begin{equation}\label{FirstDerivative-Inverse}
    (g^{ij})'= 2(\overline L_r)_{bc}g^{ib}g^{cj} = 2(\overline L_r)^{ij}.
\end{equation}
For the second derivative, we get
\[
\begin{split}
(g^{ij})'' &= -g''_{bc}g^{ib}g^{cj} - g'_{bc}(g^{ib})'g^{cj} - g'_{bc}g^{ib}(g^{cj})'\\&=-(-2\overline R_{rbrc} + 2(\overline L_r)^a_b(\overline L_r)_{ac})g^{ib}g^{cj} - (-2(\overline L_r)_{bc})(2(\overline L_r)^{ib})g^{cj} - (-2(\overline L_r)_{bc})g^{ib}(2(\overline L_r)^{cj}).
\end{split}
\]
That is,
\begin{equation}\label{SecondDerivative-Inverse}
    (g^{ij})'' = 2g^{ib}g^{jc}\overline R_{rbrc} +6(\overline L_r)^{ai}(\overline L_r)_a^j.
\end{equation}
\end{remark}

\begin{proposition}\label{MeanCurvatureExpansion-compactified}
Let $r$ be a special defining function and let $\overline g$ be the associated compactified metric. For $r>0$ small enough, we denote by $\overline H_r$ the mean curvature of the smooth hypersurface $\Sigma_r = \{x\in M: r(x) = r\}$. Then

\begin{equation}
\begin{split}
   \overline H_r &= \hat H + (|\overline L_r|^2 + \overline R_{rr})|_{r=0}\cdot r + \left(2(\overline L_r)^{ij}\overline R_{rirj} + (\overline L_r)^{ai}(\overline L_r)_a^j(\overline L_r)_{ij} + \frac{1}{2}\overline R_{rr,r}\right)|_{r=0}\cdot r^2\\&\hspace{.15in} + O(r^3),
\end{split}
\end{equation}
where $\overline R_{rr}$ denotes the Ricci curvature $(\Ric_{\overline g})_{rr}$.
\end{proposition}
\begin{proof}
The derivative of $\overline H_r$ can be computed from
\[
\begin{split}
\overline H_r' &= -\frac{1}{2}\partial_r(g^{ij}g'_{ij})\\&= -\frac{1}{2}(g^{ij})'g'_{ij} - \frac{1}{2}g^{ij}g''_{ij}\\&= -\frac{1}{2}(2(\overline L_r)^{ij})(-2(\overline L_r)_{ij}) - \frac{1}{2}g^{ij}(-2\overline R_{rirj} + 2 (\overline L_r)^a_i(\overline L_r)_{aj}).
\end{split}
\]
This gives
\begin{equation}\label{MeanCurvature-1Derivative}
\overline H_r' = |\overline L_r|^2+\overline R_{rr} 
\end{equation}
The second derivative of $\overline H_r$ is given by 
\[
\begin{split}
\overline H_r'' =\;& -\frac{1}{2}\partial_r((g^{ij})'g'_{ij}) - \frac{1}{2}\partial_r(g^{ij}g''_{ij})\\ =\;& -\frac{1}{2}(g^{ij})''g'_{ij} - (g^{ij})'g''_{ij} - \frac{1}{2}g^{ij}g'''_{ij} \\ =\;& -\frac{1}{2}(2g^{ib}g^{jc}\overline R_{rbrc} +6(\overline L_r)^{ai}(\overline L_r)_a^j)(-2(\overline L_r)_{ij}) - (2(\overline L_r)^{ij})(-2\overline R_{rirj}+2(\overline L_r)^a_i(\overline L_r)_{aj})\\&-\frac{1}{2}g^{ij}(-2\overline R_{rirj,r} + 2((\overline L_r)^k_i\overline R_{rkrj} + (\overline L_r)^k_j\overline R_{rkri})) \\ =\; & 4(\overline L_r)^{ij}\overline R_{rirj} + 2(\overline L_r)^{ai}(\overline L_r)_a^j(\overline L_r)_{ij} + \overline R_{rr,r}
\end{split}
\]
The result now follows from 
\begin{equation}
    \overline H_r = \hat H + \overline H_0'\cdot r + \frac{1}{2}\overline H''_0\cdot r^2 + O(r^3)
\end{equation}
\end{proof}

In occasions, it would be convenient to work with $H_r$, which is the mean curvature of the smooth hypersurface $\Sigma_r=\{r(p)=r\}$, $r$ small, with respect to the singular metric $g_+$. Since $g_+ = r^{-2}\overline g$, the conformal transformation law yields
\begin{equation}\label{TransformationLawMeanCurvature}
H_r = n+ r\overline H_r.
\end{equation}
This immediately gives us that $H_r$ extends all the way to $\partial M$ and equals $n$ there, that is, $H_0 = n$. The higher order terms in its expansion can be computed from (\ref{TransformationLawMeanCurvature}) and Proposition \ref{MeanCurvatureExpansion-compactified}. 

\begin{corollary}\label{MeanCurvatureExpansion-complete}
    Let $r$ be a special defining function and let $\overline g$ be the associated compactified metric. For $r>0$ small enough, denote by $H_r$ the mean curvature of $\Sigma_r$ with respect to the singular metric $g_+$. Then
    \begin{equation}
        H_r = n+\hat H\cdot r + (|\hat L|^2+ \overline R_{00})\cdot r^2 + O(r^3),
    \end{equation}
     where $\overline R_{00}=(\overline R_{rr})|_{\Sigma}$.
\end{corollary}

\begin{proof}
As we mentioned, the derivatives of $H_r$ can be computed from (\ref{TransformationLawMeanCurvature}) and Proposition \ref{MeanCurvatureExpansion-compactified}. Indeed,
\begin{equation}
H'_r = \overline H_r + r\overline H'_r,\quad\text{and}\quad H''_r = \overline H'_r + \overline H'_r+ r\overline H''_r.
\end{equation}
Therefore, 
\begin{equation}
    H'_0 = \hat H
\end{equation}
and
\begin{equation}
  H''_0 = 2\overline H'_0 = 2|\hat L|^2+ 2\overline R_{00}.  
\end{equation}
Hence, the expansion of $H_r$ is given by 
\[
\begin{split}
H_r & = n + H'_0\cdot r + \frac{1}{2}H''_0\cdot r^2+O(r^3)\\&= n+\hat H\cdot r + (|\hat L|^2+ \overline R_{00})\cdot r^2 + O(r^3).
\end{split}
\]
\end{proof}

\begin{lemma}\label{SC-BoundaryGeometry}
    Consider $(\Sigma, \hat g)$ as a submanifold of $(\overline M^{n+1}, \overline g)$ of codimension $1$. Then 
    \begin{equation}
        \overline R_{00} = \frac{1}{2}(R_{\overline g}|_{\Sigma} - R_{\hat g} - |\hat L|^2 + \hat H^2).
    \end{equation}
\end{lemma}

\begin{proof}
Working in a local adapted frame, the Gauss equation yields
\[
\overline R_{ijkl} = \hat R_{ijkl} + \hat L_{jk}\hat L_{il} - \hat L_{ik}\hat L_{jl},
\]
allowing us to express curvature quantities with respect to the compactified metric in terms of boundary information. Notice that the above formula holds only at boundary points. We will use $\overline R_{ij}$ and $\hat R_{ij}$ for $(\Ric_{\overline g})_{ij}$ and $(\Ric_{\hat g})_{ij}$.

Contracting with the boundary metric gives
\[
\hat g^{ik}\overline R_{ijkl} = \hat R_{jl} + \hat L_{j}^a\hat L_{a l} - \hat H\hat L_{jl}.
\]
Since
\[
\hat g^{ik}\overline R_{ijkl} = \overline R_{jl} - \overline g^{rr}\overline R_{irkr} = \overline R_{jl} - \overline R_{irkr},
\]
we deduce
\[
\overline R_{jl} - \overline R_{rjrl} = \hat R_{jl} + \hat L_{j}^a\hat L_{al} - \hat H\hat L_{jl}.
\]
Contracting once more with the boundary metric, gives 
\[
\hat g^{jl}(\overline R_{jl} - \overline R_{rjrl}) = \hat R + |\hat L|^2 - \hat H^2,
\]
and so 
\begin{equation}\label{BoundaryGeometry-SC}
    R_{\overline g}|_{\Sigma} -2\overline R_{00}= R_{\hat g} + |\hat L|^2 - \hat H^2. 
\end{equation}
Therefore, 
\begin{equation}\label{BoundaryGeometry-Riccirr}
\overline R_{00} = \frac{1}{2}(R_{\overline g}|_{\Sigma} - R_{\hat g} - |\hat L|^2 + \hat H^2).    
\end{equation}
\end{proof}

\begin{remark}
Notice that we could rewrite $H''_0$. Indeed, recall that using Gauss equation gives
\[
\overline R_{00} = \frac{1}{2}(R_{\overline g}|_{\Sigma} - R_{\hat g} - |\hat L|^2 + \hat H^2),
\]
and so
\[
|\hat L|^2+\overline R_{00} = \frac{1}{2}(R_{\overline g}|_{\Sigma} - R_{\hat g} + |\hat L|^2 + \hat H^2).
\]
Therefore, with the setup as in Corollary \ref{MeanCurvatureExpansion-complete}, we have 
\begin{equation}\label{MeanCurvature-SecondDerivative}
        H_r = n+\hat H\cdot r + \frac{1}{2}(R_{\overline g}|_{\Sigma} - R_{\hat g} + |\hat L|^2 + \hat H^2)\cdot r^2 + O(r^3).
    \end{equation}
\end{remark}

Finally, in computing the expansion of $\Ric_{g_+}+ng_+$, we need the expansion of the second fundamental form $\overline L_r$. We state it as our last proposition.

\begin{proposition}\label{2ff-expansion}
    Let $( M^{n+1},g_+)$ be an AH manifold and let $\overline g = r^2g_+$ be a compactification of it by a special defining function $r$. Then 
\begin{equation}\label{2ff-Expansion}
\begin{split}
    (\overline L_r)_{ij} =\;& \hat L_{ij} + (\overline R_{rirj} - (\overline L_r)_i^a(\overline L_r)_{aj})|_{r=0}\cdot r\\&+ \frac{1}{2}(\overline R_{rirj,r} -((\overline L_r)^k_i\overline R_{rkrj} + (\overline L_r)^k_j\overline R_{rkri}))|_{r=0}\cdot r^2+ O_{ij}(r^3).
\end{split}
\end{equation}
\end{proposition}
\begin{proof}
    By definition, we know $\overline L_0 = \hat L$ is the second fundamental form of $(\Sigma, \hat g)$. Its first derivative is given by 
    \[
(\overline L_r)'_{ij} = \left(-\frac{1}{2}g'_{ij}\right)' = \overline R_{rirj} - (\overline L_r)_i^a(\overline L_r)_{aj},
\]
while the second derivative is
\[
(\overline L_r)''_{ij} = \left(-\frac{1}{2}g'_{ij}\right)'' = \overline R_{rirj,r} -((\overline L_r)^k_i\overline R_{rkrj} + (\overline L_r)^k_j\overline R_{rkri}).
\]
This gives the result.
\end{proof}

\section{Weakly Poincar\'e--Einstein Manifolds with $\Ric_{g_+}+ng_+\ge 0$ }\label{WPE+Ric-Manifolds}

Our setup throughout this section follows the same structure as that outlined at the beginning of Section \ref{Preliminaries}. To prove Theorem \ref{WPE+RiciffScDecayCond}, we will require both the expansions of $\Ric_{g_+} + ng_+$ and $R_{g_+}$. To achieve this, it is necessary to compute both the Laplacian and the Hessian of the special defining function $r$ with respect to the compactified metric. In local coordinates, $\Delta_{\overline g}$ acting on a $C^2$-function is
\[
\begin{split}
    \Delta_{\overline g}f &= \overline g^{\alpha\beta}(\partial^2_{\alpha\beta} f - \overline{\Gamma}^\gamma_{\alpha\beta}\partial_\gamma f) = \overline g^{rr}f'' + \overline g^{ij}(\partial^2_{ij}f - \overline{\Gamma}_{ij}^\alpha \partial_\alpha f)\\ &= f'' + g_r^{ij}(\partial^2_{ij}f - \overline{\Gamma}_{ij}^k \partial_k f) -g_r^{ij}\overline{\Gamma}_{ij}^rf' = f'' + \Delta_{g_r}f - \overline H_r f'
\end{split}
\]
In particular,
    \begin{equation}\label{LaplaceOf-r}
    \Delta_{\overline g} r = -\overline H_r.
    \end{equation}
More generally, the components of the hessian of the special defining function with respect to the compactified metric are given by 
\begin{equation}
(\nabla^2_{\overline g} r)_{ij} = -\overline{\Gamma}_{ij}^r = -(\overline{L}_r)_{ij}\quad\text{ and }\quad (\nabla^2_{\overline g} r)_{rj} = -\overline{\Gamma}_{rj}^r= 0 = \overline{\Gamma}^r_{rr} = 
 (\nabla^2_{\overline g}r)_{rr}.
\end{equation}

We are now in position to prove the next proposition. 

\begin{proposition}\label{RicciCurvatureExpansion1}
Assume $n+1\ge 3$. Let $( M^{n+1},g_+)$ be an asymptotically hyperbolic manifold, and let $\overline g = r^2g_+$ be a compactification of it by a special defining function $r$. The expansion of $(\Ric_{g_+} + ng_+)_{ij}$ on $\Sigma\times [0,\epsilon_o)_r$ at $r=0$ is given by
\begin{equation}\label{RicciCurvatureExpansion2}
    (-(n-1)\hat L_{ij} - \hat H \hat g_{ij})\cdot r^{-1} + (\overline R_{ij} -(n-1)(\overline R_{0i0j} - (\hat L)^a_i (\hat L)_{aj}) - (|\hat L|^2+\overline R_{00})\hat g_{ij}) + O_{ij}(r).
\end{equation}
On the other hand, the expansion of $(Ric_{g_+} + ng_+)_{rr}$ at $r=0$ is given by 
\begin{equation}\label{RicciCurvatureExpansion-3}
 -\hat H\cdot r^{-1} -|\hat L|^2 + O(r).
\end{equation}
Furthermore, (\ref{RicciCurvatureExpansion2}) can be rewritten as
\begin{equation}\label{RicciCurvatureExpansion4}
\begin{split}
    &\left(-(n-1)\mathring{\hat L} - \frac{2n-1}{n}\hat H\hat g\right)\cdot r^{-1} \\&\hspace{.5in}+\left(-(n-1)\overline{W}_{0i0j} + \left(\hat R - \frac{n-1}{n}\overline{R}\right)\overline g_{ij}  + (n-1)(\hat L)^a_i (\hat L)_{aj} - \hat H^2\cdot\hat g_{ij}\right)+O_{ij}(r)
\end{split}
\end{equation}
\end{proposition}

\begin{proof}
Recall that we use $g_{ij}$ for $(g_r)_{ij}$. Under the conformal change $\overline g\mapsto g_+ = r^{-2}\overline g$, the Ricci curvature transforms as 
\begin{equation}\label{TransformationLawRicci}
    \begin{split}
    \Ric_{g_+} &= \Ric_{\overline g} - (n-1)((\nabla^2_{\overline g}(-\ln r)) - r^{-2}dr\otimes dr)-(\Delta_{\overline g}(-\ln r) + r^{-2}(n-1)|dr|^2_{\overline g})\overline g\\ &= \Ric_{\overline g} -(n-1)((\nabla^2_{\overline g}(-\ln r)) - r^{-2}dr\otimes dr) - (r^{-1}\overline H_r +nr^{-2})\overline g
    \end{split}
    \end{equation}
Recall that we use $R_{\alpha\beta}$ and $\overline R_{\alpha\beta}$ to denote the components of $\Ric_{g_+}$ and $\Ric_{\overline g}$, respectively. Starting with the tangential components, we derive
\begin{equation}\label{TransformationLawRicciTangential}
    \begin{split}
    R_{ij} &= \overline R_{ij} - (n-1)(r^{-1}\overline \Gamma_{ij}^r) - (r^{-1}\overline H_r + nr^{-2})\overline g_{ij} \\ &= \overline R_{ij}-(n-1)r^{-1}(\overline L_r)_{ij} - (r^{-1}\overline H_r + nr^{-2}) \overline g_{ij}.
    \end{split}
    \end{equation}
On the other hand,
    \begin{equation}\label{TransformationLawRicciNormal}
    \begin{split}
        R_{rr} &= \overline R_{rr} -(n-1)r^{-2}+ (n-1)r^{-2} - r^{-1}\overline H_r - nr^{-2}\\&=\overline R_{rr}-r^{-1}\overline H_r - nr^{-2}.
    \end{split}
    \end{equation}
    Rewriting  (\ref{TransformationLawRicciTangential}) and (\ref{TransformationLawRicciNormal}), we obtain 
    \begin{equation}
         (\Ric_{g_+}+ng_+)_{ij} = (\Ric_{\overline g})_{ij}-(n-1)r^{-1}(\overline L_r)_{ij} - r^{-1}\overline H_r g_{ij}
    \end{equation}
    and 
    \begin{equation}
        (\Ric_{g_+}+ng_+)_{rr} =(\Ric_{\overline g})_{rr}-r^{-1}\overline H_r,
    \end{equation}
    where we have used that $\overline g_{ij} = g_{ij}$.
    Expansions (\ref{RicciCurvatureExpansion2}) and (\ref{RicciCurvatureExpansion-3}) now follow from Proposition \ref{MeanCurvatureExpansion-compactified}, Proposition \ref{2ff-expansion} and Proposition \ref{ExpansionMetricOrder3}.

    We proceed with the proof of (\ref{RicciCurvatureExpansion4}). Notice that 
\[
\begin{split}
\overline{R}_{0i0j} &= \overline{W}_{0i0j} + \overline{P}_{ij} + \overline{P}_{00}\overline g_{ij} \\&=\overline{W}_{0i0j} + \frac{\overline R_{ij}}{n-1} - \frac{\overline R}{2n(n-1)}\overline g_{ij} + \frac{\overline R_{00}}{n-1}\overline g_{ij} - \frac{\overline R}{2n(n-1)}\overline g_{ij},
\end{split}
\]
thus
\[
\overline{R}_{ij}-(n-1)\overline{R}_{0i0j}  = -(n-1)\overline{W}_{0i0j} + \frac{\overline R}{n}\overline g_{ij} - \overline{R}_{00}\overline g_{ij}.
\]
Recall that thanks to Lemma \ref{SC-BoundaryGeometry} we have $-2\overline{R}_{00} = -\overline{R} + \hat R + |\hat L|^2 - \hat H^2$ on $\Sigma$. Therefore, the zeroth order term in the expansion of $(\Ric_{g_+}+ng_+)_{ij}$ can be rewritten as
\[
\begin{split}
&-(n-1)\overline{W}_{0i0j} + \frac{\overline R}{n}\overline g_{ij} - 2\overline{R}_{00}\overline g_{ij} + (n-1)(\hat L)^a_i (\hat L)_{aj} - |\hat L|^2\hat g_{ij}\\&=-(n-1)\overline{W}_{0i0j} + \left(\hat R - \frac{n-1}{n}\overline{R}\right)\overline g_{ij}  + (n-1)(\hat L)^a_i (\hat L)_{aj} - \hat H^2\hat g_{ij}
\end{split}
\]
As for the coefficient of $r^{-1}$, notice that $\mathring{\hat L} = \hat L - \frac{\hat H}{n}\hat g$, thus 
\[
-(n-1)\hat L - \hat H\hat g = -(n-1)\mathring{\hat L} - \frac{2n-1}{n}\hat H\hat g
\]
This finishes the proof.
\end{proof}

We would like to point out that higher order terms of the Ricci tensor could be computed using Proposition \ref{MeanCurvatureExpansion-compactified}, Proposition \ref{2ff-expansion} and Proposition \ref{ExpansionMetricOrder3}, but we would not need it in the current work. Let us now move onto the asymptotics of the scalar curvature. We again look at the transformation law under conformal changes for $R_{g_+}$ plus the expansion of the mean curvature vector.

\begin{proposition}\label{ScalarCurvature-Expansion}
    Assume $n+1\ge 3$. Let $( M^{n+1},g_+)$ be an asymptotically hyperbolic manifold, and let $\overline g = r^2g_+$ be a compactification of it by a special defining function $r$. The expansion of $R_{g_+}+n(n+1)$ on $\Sigma\times [0,\epsilon_o)_r$ at $r=0$ is given by 
    \begin{equation}
        R_{g_+} +n(n+1) =[-2n\hat H]\cdot r + [-(n-1)R_{\overline g}|_{\Sigma} +nR_{\hat g} -n(|\hat L|^2 + \hat H^2)]\cdot r^2 + O(r^3).
    \end{equation}
\end{proposition}

\begin{proof}
    In order to compute this expansion, we first conformally change the metric by $g_+\mapsto r^{-2}\overline g$ and obtain 
\[
\begin{split}
R_{g_+} & = g_+^{\alpha\beta}(\Ric_{g_+})_{\alpha\beta} = r^2\overline g^{\alpha\beta}(\Ric_{g_+})_{\alpha\beta} = r^2\overline g^{ij}(\Ric_{g_+})_{ij} + r^2\overline g^{rr}(\Ric_{g_+})_{rr} \\&= r^2\overline g^{ij}(\overline R_{ij}-(n-1)r^{-1}(\overline L_r)_{ij} - (r^{-1}\overline H_r + nr^{-2}) \overline g_{ij})+r^2(\overline R_{rr}-r^{-1}\overline H_r - nr^{-2}) \\ &= r^2\overline R - r(n-1)\overline H_r -nr^2(r^{-1}\overline H_r + nr^{-2}) - r\overline H_r - n = r^2\overline R -2nr\overline H_r -n(n+1)
\end{split}
\]
that is, 
\begin{equation}\label{ScalarCurvature-ConformalChange}
R_{g_+} = r^2\overline R -2nr\overline H_r -n(n+1). 
\end{equation}
Using Proposition \ref{MeanCurvatureExpansion-compactified} and Lemma \ref{SC-BoundaryGeometry}, we deduce
\[
\begin{split}
R_{g_+}+n(n+1) &= r^2 \overline R - 2nr( \hat H + (|\overline L_r|^2 + \overline R_{rr})|_{r=0}\cdot r+O(r^2)) \\&= r^2(\overline R|_{\Sigma} + O(r)) - 2nr\hat H -nr^2 (\overline R|_{\Sigma} - \hat R + |\hat L|^2 + \hat H^2) + O(r^3),
\end{split}
\]
and the result follows.
\end{proof}

As an immediate consequence of Proposition \ref{ScalarCurvature-Expansion},
\begin{equation}\label{ScalarCurvature-rsquared}
\frac{R_{g_+} +n(n+1)}{r^2} =[-2n\hat H]\cdot r^{-1}+ [-(n-1)\overline R|_\Sigma +n\hat R -n(|\hat L|^2 + \hat H^2)] + O(r)
\end{equation}
Therefore, $R_{g_+}+n(n+1) = o(r^2)$ if and only if $\hat H = 0$, that is, if $\Sigma$ is minimal, and  \[-(n-1)\overline R|_\Sigma +n\hat R -n|\hat L|^2=0.\] In this case, from (\ref{MeanCurvature-SecondDerivative}) we obtain
\[
H''_0 = \overline R|_\Sigma - \hat R + |\hat L|^2 = \frac{n-(n-1)}{n-1}\hat R + \frac{(n-1)-n}{n-1}|\hat L|^2  = \frac{1}{n-1}\hat R -\frac{1}{n-1}|\hat L|^2.
\]
Therefore, 
\begin{equation}\label{MeanCurvature-2ndDerivative}
H''_0 \le \frac{1}{n-1}\hat R, 
\end{equation}
with equality if and only if $|\hat L|^2 = 0$, that is, if the boundary is totally geodesic. We would like to remark that inequality (\ref{MeanCurvature-2ndDerivative}) is derived in \cite{HijaziOussama2020TCCo} (see Proposition 2), but the equality case is not discussed. Finally, if we further assume that $\Ric_{g_+}+ng_+\ge 0$ in $\Sigma\times [0,\epsilon_o)_r$, then it follows from (\ref{RicciCurvatureExpansion-3}) that $\Sigma$ is totally geodesic, and so $\hat R = \frac{n-1}{n} \overline R|_{\Sigma}$. 

The next two propositions contain the proof of Theorem \ref{WPE+RiciffScDecayCond}.
\begin{proposition}\label{WPEImply}
    Assume $n+1>3$. Let $(M^{n+1},g_+)$ be an AH manifold, and let $\overline g = r^2g_+$ be a compactification of it by a special defining function $r$. If $g_{(1)} = 0$ and $g_{(2)} = -\hat P$, that is, if $(M^{n+1},g_+)$ is WPE, then $R_{g_+}+n(n+1) = o(r^2)$. 
\end{proposition}

\begin{proof}
    Recall that, thanks to Proposition \ref{ExpansionMetricOrder3}, we have
    \[
    g_{(1)} = -2\hat L \quad\text{ and }\quad (g_{(2)})_{ij} = -\overline R_{0i0j} + (\hat L)_i^a \hat L_{aj}.
    \]
    It is clear then that $g_{(1)} = 0$ if and only if $\Sigma$ is totally geodesic. Therefore, (\ref{ScalarCurvature-rsquared}) gives
    \[
    \frac{R_{g_+} +n(n+1)}{r^2} = [-(n-1)\overline R|_\Sigma +n\hat R] + O(r),
    \]
    thus it is sufficient to show that $(n-1)\overline R = n\hat R$ at boundary points. 
    
    To this end, we proceed as follows. From Gauss' equation and the fact that the boundary is totally geodesic, we deduce
    \[
        \overline R_{ij} = \overline g^{ab}\overline R_{aibj} + \overline R_{0i0j}= \overline g^{ab}\hat R_{aibj}+\overline R_{0i0j} = \hat R_{ij} + \overline R_{0i0j},
    \]
    that is,
    \begin{equation}\label{RicciDifference}
        \overline R_{0i0j} = \overline R_{ij} - \hat R_{ij}.
    \end{equation}
    Therefore, from our assumption on $g_{(2)}$ we obtain $\overline R_{ij} - \hat R_{ij} = \hat P_{ij}$ and so its $\hat g$-trace gives us
    \[
    \hat g^{ij}\overline R_{ij} - \hat R = \frac{1}{2(n-1)}\hat R \iff \overline R - \overline R_{00} = \frac{2n-1}{2(n-1)}\hat R.
    \]
    The result now follows from Lemma \ref{SC-BoundaryGeometry}.
\end{proof}

\begin{proposition}\label{HMRImplyWPE}
    Assume $n+1>3$. Let $(M^{n+1},g_+)$ be an AH manifold, and let $\overline g = r^2g_+$ be a compactification of it by a special defining function $r$. If $\Ric_{g_+}+ng_+\geq 0$ in $\Sigma\times [0,\epsilon_o)_r$ and $R_{g_+}+n(n+1)=o(r^2)$, then $g_{(1)} = 0$ and $g_{(2)} = -\hat P$, that is, $(M^{n+1},g_+)$ is WPE.
\end{proposition}

\begin{proof}
From the discussion succeeding Proposition \ref{ScalarCurvature-Expansion}, we have that $\hat L = 0$ and $\hat R = \frac{n-1}{n}\overline R|_{\Sigma}$. It then follows from Lemma \ref{SC-BoundaryGeometry} that 
\begin{equation}
    \overline R_{00} = \frac{1}{2}(\overline R - \hat R) = \frac{1}{2n}\overline R
\end{equation}
at boundary points, thus 
\begin{equation}
    \overline P_{00} = \frac{1}{n-1}\overline R_{00} - \frac{\overline R}{2n(n-1)} = 0.
\end{equation}
Recall from Proposition \ref{ExpansionMetricOrder3} that
$(g_{(2)})_{ij}=-\overline{R}_{0i0j}$, where we have used that the boundary is totally geodesic. Decomposing the Riemann curvature tensor and using $\overline P_{00}=0$, we get
\begin{equation}\label{HMRImplyWPE-2}
\overline{R}_{0i0j}=\overline{W}_{0i0j}+\overline{P}_{ij}.
\end{equation}
To relate $\overline P_{ij}$ and $\hat P_{ij}$ we use the Fialkow--Gauss equation (\ref{Fialkow--GaussEq2}), where our submanifold is $(\Sigma,\hat g)$ and $(\overline{M}^{n+1},\overline g)$ is our bulk manifold. Since $\Sigma$ is a totally geodesic hypersurface, we get
\begin{equation}\label{HMRImplyWPE-3}
\overline P_{ij} = \hat P_{ij} + \frac{1}{n-2}\overline W_{0i0j},
\end{equation}
and so
\begin{equation}\label{HMRImplyWPE-1}
    \overline R_{0i0j} = \frac{n-1}{n-2}\overline W_{0i0j} + \hat P_{ij}.
\end{equation}
The proof would be completed if we can show that $\overline{W}_{0i0j}$ vanishes. To this end, notice that 
\begin{equation}
\overline{R}_{ij}=(n-1)\overline{P}_{ij}+\frac{\overline{R}}{2n}\overline{g}_{ij} = (n-1)\hat P_{ij} +\frac{n-1}{n-2}\overline{W}_{0i0j} + \frac{\overline R}{2n}\hat g_{ij}
\end{equation}
Since 
\[
\overline{R}_{ij} - (n-1)\overline{R}_{0i0j} - \overline{R}_{00}\hat{g}_{ij} = -\frac{n(n-1)}{n-2}\overline{W}_{0i0j}
\]
and has a sign thanks to Proposition \ref{RicciCurvatureExpansion1} and our Ricci curvature assumption, we deduce that $\overline{W}_{0i0j}$ has a constant sign as well. However, since \(\overline{W}_{0i0j}\) is symmetric and trace-free, it cannot have eigenvalues all non-negative or all non-positive unless they are all zero. Therefore, we conclude that \(\overline{W}_{0i0j} \equiv 0\), as desired.
\end{proof}
\begin{proposition}\label{3.4'}
    
    If, in \cref{HMRImplyWPE} $\Sigma$ is umbilic, then the condition $Ric_{g_+}+ng_+\geq 0$ may be replaced with $Ric_{g_+}+ng_+\leq 0$ and the conclusion still holds.
\end{proposition}\begin{proof}
The same argument as used in the proof of \cref{HMRImplyWPE} applies.
\end{proof}
The proof of Theorem \ref{WPE+RiciffScDecayCond} now follows from Proposition \ref{WPEImply} and Proposition \ref{HMRImplyWPE}. 

\begin{proof}[Proof of Theorem \ref{WPE-Characterization}, part $(1)$]
    Recall from Proposition \ref{ExpansionMetricOrder3} that $g_{(1)} = -2\hat L$. This means that $g_{(1)}= 0$ if and only if $\Sigma$ is totally geodesic. From Proposition \ref{RicciCurvatureExpansion1}, $(\Ric_{g_+}+ng_+)_{rr} = O(1)$ if and only if $\hat H = 0$, while $(\Ric_{g_+}+ng_+)_{ij} = O(1)$ if and only if \[-(n-1)\mathring{\hat L}_{ij} = \frac{2n-1}{n}\hat H \hat g_{ij}.\] Therefore, $g_{(1)} = 0$ if and only if $\Ric_{g_+}+ng_+ = O(1)$, as desired. In order to show the last equivalence, notice from Proposition \ref{ScalarCurvature-Expansion} that $R_{g_+}+n(n+1) = o(r)$ if and only if $\hat H = 0$, and so it follows that $\Sigma$ is totally geodesic (i.e. $g_{(1)} = 0$) if and only if $R_{g_+}+n(n+1) = o(r)$ and $\Sigma$ is umbilic. This finishes the proof. 
\end{proof}

\begin{proof}[Proof of Theorem \ref{WPE-Characterization}, part $(2)$]
    If $g_+$ is WPE, then $\Sigma$ is totally geodesic and, by Proposition \ref{WPEImply}, $R_{g_+}+n(n+1) = o(r^2)$. Therefore, $\hat R = \frac{n-1}{n}\overline R$ at boundary points as pointed out in the discussion succeeding Proposition \ref{ScalarCurvature-Expansion}. It now follows from Proposition \ref{RicciCurvatureExpansion1} that $(\Ric_{g_+}+ng_+)_{rr} = O(r)$ and 
    \[
    \begin{split}
    (\Ric_{g_+}+ng_+)_{ij} &= -(n-1)\overline W_{0i0j} + O(r).
    \end{split}
    \]
    We will be done with the first implication if we can show that $\overline W_{0i0j} = 0$. Indeed, $\overline W_{0i0j} = \overline R_{0i0j} - \overline P_{ij}$, where we have used that $\overline P_{00} = 0$ (see beginning of the proof of Proposition \ref{HMRImplyWPE}). However, Proposition \ref{ExpansionMetricOrder3} and the WPE condition gives $\hat P_{ij} = \overline R_{0i0j}$. The Fialkow--Gauss (\ref{Fialkow--GaussEq2}) then gives $\overline W_{0i0j} = \hat P_{ij} - \overline P_{ij} = -\frac{1}{n-2}\overline W_{0i0j}$, concluding the argument.
    
    Let us now assume that $\Ric_{g_+}+ng_+ = O(r)$. Then it follows from Proposition \ref{RicciCurvatureExpansion1} and the expansion of $(\Ric_{g_+}+ng_+)_{rr}$ that $\Sigma$ is totally geodesic, therefore umbilic. If we look at the tangential components $(\Ric_{g_+}+ng_+)_{ij}$, then it is $O(r)$ if and only if
    \[
    -(n-1)\overline{W}_{0i0j} + \left(\hat R - \frac{n-1}{n}\overline{R}\right)\overline g_{ij} = 0
    \]
    on the boundary. Tracing with $\hat g$ gives us $n\hat R - (n-1)\overline R = 0$ on $\Sigma$, and so $\overline{W}_{0i0j}$ vanishes as well on $\Sigma$. The condition on the scalar curvature now follows from the discussion succeeding Proposition \ref{ScalarCurvature-Expansion}. 
    
    Finally, if $R_{g_+}+n(n+1)=o(r^2)$, $\Sigma$ is umbilic and $\overline{W}_{0i0j}$ vanishes on $\Sigma$, then $\hat H=0$ and $\Sigma$ is in fact totally geodesic, so $g_{(1)} =0$. Now, from the discussion succeeding Proposition \ref{ScalarCurvature-Expansion}, we have $-(n-1)\overline R +n\hat R =0$ on $\Sigma$, which gives $\overline{P}_{00}=0$ on $\Sigma$ (see beginning of the proof for Proposition \ref{HMRImplyWPE}). By the decomposition formula for the Riemann tensor, we then deduce $\overline{R}_{0i0j} =\overline{P}_{ij}$,
    where we have used that $\overline{W}_{0i0j}$ vanishes on $\Sigma$. Since $g_{(2)} = -\overline{R}_{0i0j}$ by Proposition \ref{ExpansionMetricOrder3}, it remains to show that $\overline{P}_{ij} = \hat P_{ij}$ on $\Sigma$. However, this follows from the Fialkow--Gauss equation (\ref{Fialkow--GaussEq2}). This finishes the proof. 
\end{proof}

\section{Cheeger Constant on Asymptotically CMC Submanifolds}\label{CheegerConstantOnSubmanifolds}
We start with the proof of Theorem \ref{CHCONSTANT-UpperBound}.
\begin{proof}[Proof of Theorem \ref{CHCONSTANT-UpperBound}]
    First, from the authors work in \cite{pérezayala-tyrrell} (see Corollary 1.9) we know that
    \[
    \lambda_{1,p}(Y^{k+1})\le \left(\frac{k}{p}\right)^p\left(1-\frac{C^2}{(k+1)^2}\right)^{\frac{p}{2}}.
    \]
    On the other hand, it follows from the work of Takeuchi in \cite{Takeuchi} that 
    \[
    \lambda_{1,p}(Y^{k+1})\ge \left(\frac{\Ch(Y^{k+1})}{p}\right)^p.
    \]
    The result now follows.
\end{proof}

We now turn into the proof of Theorem \ref{LowerBoundCheegerConstant-Submanifold}. The argument is inspired by the methods used in \cite{HijaziOussama2020TCCo}. Recall that a Lee-eigenfunction on an AH manifold $(M^{n+1},g_+)$ is a smooth, positive solution $u$ of 
\begin{equation}
\begin{cases}
    \Delta_{g_+}u = (n+1)u\\ u - r^{-1} = O(1).
\end{cases}
\end{equation}
 It is known that if the AH manifold is WPE, satisfies $\Ric_{g_+}+ng_+\ge 0$, and its conformal infinity has non-negative Yamabe invariant, then a Lee-eigenfunction satisfying the gradient estimate $|\nabla_{g_+}u|^2_{g_+} \le u^2$ exists; see Remark 1.3 in \cite{GuillarmouColin2010SCoP}.

\begin{proof}[Proof of Theorem \ref{LowerBoundCheegerConstant-Submanifold}]
    Let $u$ be a Lee-eigenfunction on the bulk manifold $(M^{n+1},g_+)$ and recall that $b(u) = \nabla^2_{g_+}u - ug_+$. Set $\hat u = u|_Y$, $\hat f = \ln \hat u$ and $T(u) = \Tr_{g_+}((\nabla^2_{g_+}u)|_{(TY^2)^\perp})$. A quick calculation shows that 
    \[
    T(u) = \Tr_{g_+}b(u)|_{(TY^2)^\perp} + \Tr_{g_+}(\hat ug_+)|_{(TY^2)^\perp} = \Tr_{g_+}b(u)|_{(TY^2)^\perp}+(n-k)\hat u.
    \]
    Therefore, we have
    \[
    \begin{split}
    \Delta_{h_+}\hat f & = \hat u^{-1}\Delta_{h_+}\hat u - \hat u^{-2}|\nabla_{h_+}\hat u|^2 \\&=\hat u^{-1}\left\{(\Delta_{g_+}u)|_Y + H^Y(u) - T(u)\right\} - \hat u^{-2}|\nabla_{h_+}\hat u|^2 \\&= \hat u^{-1}\left\{(n+1)\hat u + H^Y(u)-\Tr_{g_+}b(u)|_{(TY^2)^\perp} - (n-k)\hat u\right\} - \hat u^{-2}|\nabla_{h_+}\hat u|^2 \\ &= (k+1) + \hat u^{-1}H^Y(u) - \hat u^{-1}\Tr_{g_+}b(u)|_{(TY^2)^\perp} - \hat u^{-2}|\nabla_{h_+}\hat u|^2.
    \end{split}
    \]
    Using 
    \begin{equation}\label{GradientEstimate}
    |\nabla_{h_+}\hat f|^2 = \hat u^{-2}|\nabla_{h_+}\hat u|^2\le \hat u^{-2}|\nabla_{g_+}u|^2\le 1,
    \end{equation}
    we then deduce
    \[
        \Delta_{h_+}\hat f \ge k - \beta^Y(u)+ \hat u^{-1}H^Y(u).
    \]
    To estimate the last term, we use Cauchy-Schwarz inequality and our assumption on the mean curvature to get
    \[
    |\hat u^{-1} H^Y(u)| \le \alpha,
    \]
    thus
    \begin{equation}\label{Estimate-TestFunction}
        \Delta_{h_+}\hat f \ge k-\beta^Y(u)-\alpha.
    \end{equation}

    We are in position to estimate the Cheeger constant of $Y^{k+1}$. Let $\Omega\subseteq Y$ be a bounded domain with smooth boundary $\partial \Omega$, and denote by $N_\Omega$ its inward unit normal. Integrating by parts gives 
    \[
    \int_{\Omega} \Delta_{h_+}\hat f\; dv_{h_+} = -\int_{\partial \Omega}\langle \nabla_{h_+}\hat f, N_\Omega\rangle\; d\sigma_{h_+} \le A(\partial \Omega),
    \]
    where we have used the gradient estimate (\ref{GradientEstimate}).
    On the other hand, integrating (\ref{Estimate-TestFunction}) yields
    \[
    \int_{\Omega} \Delta_{h_+}\hat f \;dv_{h_+} \ge (k-\beta^Y(u) - \alpha)V(\Omega).
    \]
    This finishes the proof.
\end{proof}

\section{Examples: Proof of Theorem \ref{EX-PosNegYamabe} and Theorem \ref{WPE-Submanifolds}}\label{Examples}

\subsection{Non-Einstein AH Manifolds with $\Ch = n$ and $\mathcal{Y}>0$.}

In this section, we provide a class of submanifolds of hyperbolic space which are asymptotically minimal. In fact, these will be a class that comes from rotation hypersurfaces which are minimal and are known in literature as catenoids. For the purpose of this section, we change from the Poincar\'e ball model to the hyperboloid model as it is our intention to stay closer to Do Carmo--Dajczer's presentation in \cite{do2012rotation}. To this end, consider $\mathbb{R}^{n+3}$ equipped with the Lorentzian metric 
\[
g_{-1} = -dx_1^2+\sum_{i=2}^{n+3}dx_i^2.
\]
Then the subset $\{x\in\mathbb{R}^{n+3}: g_{-1}(x,x) = -1,x_1>0\}$, equipped with the induced metric $g_H$ from $(\mathbb{R}^{n+3},g_{-1})$, is $\mathbb{H}^{n+2}(-1)$.

We proceed with a brief description of Do Carmo--Djaczer's construction. Let $P^2$ denote a two dimensional subspace in $\mathbb{R}^{n+3}$, and denote by $G$ the subgroup of orthogonal transformations with positive determinant of $(\mathbb{R}^{n+3},g_{-1})$ which leave $P^2$ pointwise fixed. We now pick a three dimensional subspace $P_3$ of $\mathbb{R}^{n+3}$ containing $P^2$ and with $P^3\cap \mathbb{H}^{n+2}(-1)\not= \emptyset$, and we also select a smooth curve $\gamma$ in $P^3\cap \mathbb{H}^{n+2}(-1)$ which does not intersect $P^2$. The orbit of $\gamma$ under $G$ is called a rotation hypersurface and $\gamma$ is known as the generating curve. Let us denote this rotation hypersurface by $\C^{n+1}$. If 
the initial two dimensional subspace $P^2$ is such that $g_{-1}|_{P^2}$ is a Lorentzian metric, then $\C^{n+1}$ is called a spherical rotation hypersurface; if $g_{-1}|_{P^2}$ is a Riemannian metric, then $\C^{n+1}$ is called a parabolic rotation hypersurface; and if $g_{-1}|_{P^2}$ is a degenerate quadratic form, then $\C^{n+1}$ is called a hyperbolic rotation hypersurface.

In \cite{do2012rotation}, it is shown that a parametrization of these rotation hypersurfaces exists which depends only on a single function. It is shown that minimality boils down to solving a nonlinear ODE. Indeed, if $s$ denotes the arc length of the generating curve $\gamma$ and if we assume, up to a rotation and a choice of orthonormal basis, that $\gamma(s) = (x_1(s),0,\cdots,0,x_{n+2}(s),x_{n+3}(s))$, then the rotation hypersurface $\C^{n+1}$ is minimal if and only if 
\begin{equation}\label{MinimalODE}
x_1x_1'' + n(x_1')^2 - (n+1)x_1^2 - \delta n=0;
\end{equation}
see equation $(3.13)$ in \cite{do2012rotation}\footnote{The discrepancy comes from the fact that our hypersurfaces are $(n+1)$ dimensional instead of $n$ dimensional.}. Here, $\delta = 1$ if $\C^{n+1}$ is spherical, $\delta = 0$ if $\C^{n+1}$ is parabolic, and $\delta = -1$ if $\C^{n+1}$ is hyperbolic. In any such case, a minimal rotation hypersurface  $\C^{n+1}$ is simply called a catenoid. Spherical, parabolic and hyperbolic catenoids (solutions to (\ref{MinimalODE})) have been proven to exist \cites{do2012rotation, wang2018simons}, and in many cases they have been shown to be embedded; see Theorem 4.(ii), Theorem 9., part 3, and Theorem 10, part 2, in the work of Palmas  \cite{PalmasOscar1999Crhw} for a full description and classification in different cases.

We focus on the spherical catenoids as these are conformally compact minimal hypersurfaces. Let $(\C^{n+1}, h$) be a spherical catenoid in $\mathbb{H}^{n+2}(-1)$, where $h$ denotes the induced metric. We summarize the basic properties each of these satisfy:
\begin{enumerate}
    \item They are examples of minimal conformally compact submanifolds of $\mathbb{H}^{n+2}(-1)$. Their boundary has two connected components, which, up to an isometry, is just the union of two round $n$-dimensional spheres. In particular, their conformal infinity is of positive Yamabe type. Moreover, these submanifolds are asymptotically hyperbolic themselves as a consequence of Proposition 1.8 in \cite{pérezayala-tyrrell}. 
    \item Spherical catenoids are not totally geodesic. This can be seen using the formulas for the principal curvatures given in Proposition 3.2 in \cite{do2012rotation}. The following two formulas follow from Gauss equation; see Appendix \ref{SubmanifoldGeometry} for a proof:
    \begin{equation}\label{Ricci-MinimalSubmanifold}
    \Ric^\C = -nh - B^2,
    \end{equation}
    and
    \begin{equation}\label{Scalar-MinimalSubmanifold}
    R^\C = -n(n+1) - |B|^2_{g_H}.
    \end{equation}
    Here $B$ denotes the second fundamental form of $\C^{n+1}$, and $B^2$ is a positive definite symmetric two tensor whose components are $(B^2)_{\alpha\gamma} = g_H^{\beta\delta}B_{\alpha\delta}B_{\beta\gamma}$. Therefore, spherical catenoids are not Einstein as they are not totally geodesic.
    
    \item Thanks to Theorem \ref{CheegerConstant-SubHyperbolicSpaceI}, we have that $\Ch(\C^{n+1},h) = n$.
\end{enumerate}

The previous discussion proves Theorem \ref{EX-PosNegYamabe}, part (1). Note that there are many more examples satisfying part (1) of \ref{EX-PosNegYamabe} but the spherical catenoids have additional interesting properties.

\subsection{Non-Einstein AH Manifolds with $\Ch = n$ and $\mathcal{Y}<0$.}

The key ingredient in the construction of these examples is the existence of smooth solutions to the so called Asymptotic Plateau Problem in hyperbolic space; see \cite{CoskunuzerBaris2009APP} for a beautiful survey on the topic. Given any $k$-dimensional submanifold $\Gamma^k$ in $(\mathbb{S}^{n+1},g_{round})$, the asymptotic Plateau problem asks for the existence of a minimal submanifold $Y^{k+1}$ of $(\mathbb{H}^{n+2}(-1),g_H)$ satisfying $\partial Y = \Gamma$. Any such minimal submanifold $Y^{k+1}\subset \mathbb{H}^{n+2}(-1)$ arising as a solution to the asymptotic Plateau problem would be an example of a conformally compact minimal submanifold, provided $\Gamma$ is regular enough, and so our results would apply. Furthermore, they would be AH spaces themselves thanks to Proposition 1.8 in \cite{pérezayala-tyrrell}. 

Anderson showed in \cite{AndersonMichaelT.1982Cmvi} that given a closed hypersurface $\Gamma$ in $(\mathbb{S}^{n+1},g_{round})$ there exists a complete area minimizing $n+1-$rectifiable current $Y^{n+1}$ in $\mathbb{H}^{n+2}(-1)$ that is asymptotically bounded by $\Gamma$. By geometric measure theory \cite{federer2014geometric} it is known that $Y^{n+1}$ and smooth and embedded in the interior away from a possible singular set of Hausdorff dimension at most $n-6.$ For minimal submanifolds in spaceforms the regularity theory actually implies that $Y$ is real-analytic away from the singular set. In particular $Y^{n+1}$ is real-analytic on its interior if $1\leq n \leq 5.$ Any such $Y^{n+1}$ given by Anderson's result is an example of a minimal filling for the given boundary data $\Gamma^{n}$. These are minimal submanifolds, and are asymptotic to the given boundary. Examples with conformal infinity of negative Yamabe type and Cheeger constant equal to $n$ can now be constructed as follows, at least for $2\leq n\leq 5$. Pick an embedded $C^{3,\alpha}$ closed hypersurface $\Gamma^{n}$ in $(\mathbb{S}^{n+1},g_{round})$ with negative Yamabe constant and consider its minimal filling $Y^{n+1}$ inside $\mathbb{H}^{n+2}(-1)$. These minimal fillings satisfy the following properties:

\begin{enumerate}
\item Applying boundary regularity results of Lin and$\backslash$or Tonegawa \cite{lin1989dirichlet, tonegawa1996existence} it follows that, assuming $\Gamma$ is $C^{3,\alpha}$, each minimal filling $Y^{n+1}$ has a completion $\overline{Y}^{n+1}$ which is a $C^{3,\alpha}$ manifold with boundary. By section 4 of \cite{pérezayala-tyrrell} it follows that the angle $\theta$ between $\nabla^{\overline{g}}r$ and $\overline{N}$ at a point on $\overline{Y}$ satisfies
\begin{equation}
\cos(\theta)=\frac{H}{n+1}+O(r),
\end{equation}
therefore these hypersurfaces meet the ideal boundary of hyperbolic space at right angles, in particular they meet it transversely. Hence these are examples of conformally compact minimal submanifolds and thus Theorem \ref{CheegerConstant-SubHyperbolicSpaceI} gives us that $\Ch(Y^{n+1}) = n$.

\item Since $Y^{n+1}$ is a submanifold inside $\mathbb{H}^{n+2}(-1)$ whose boundary $\Gamma^n$ is of negative Yamabe type, we conclude that $Y^{n+1}$ is not totally geodesic. In fact, the only totally geodesic submanifolds in $\mathbb{H}^{n+2}(-1)$ are lower dimensional copies of hyperbolic space, which are asymptotically bounded by round spheres (see Chapter 7 in \cite{Spivak}). Hence, conformally compact minimal $Y^{n+1}$ with $\partial Y^{n+1} = \Gamma^n$ being of negative Yamabe type are non-Einstein, as \ref{Ricci-MinimalSubmanifold} shows. 
\item Moreover, thanks to Theorem \ref{WPE-Submanifolds}, if $\Gamma^n$ is not a union of round spheres (i.e. non-umbilic) in $(\mathbb{S}^{n+1},g_{round})$, then $(Y^{n+1}, h_+)$ is not WPE. 
\end{enumerate}

 Take $\mathbb{S}^3$, for instance. One may conformally map any compact, orientable, genus $\gamma$ surface from $\mathbb{R}^3$ into $\mathbb{S}^3$ by stereographic projection. These will have negative Yamabe invariant for $\gamma>1$ by the Gauss--Bonnet theorem. Better yet, it follows by a result of Lawson \cite{LawsonH.Blaine1970CMSi} that for any given genus $\gamma$, there exists at least one  embedded closed minimal surface $\Gamma^2$ in $\mathbb{S}^3$ with that given genus. Different examples of embedded minimal surfaces with negative total Gauss curvature in $\mathbb{S}^3$ are also given by Karcher--Pinkall--Sterling in \cite{KarcherH.1988Nmsi}; see \cite{BrendleSimon2013Msi:} for a nicely written survey on the topic. 

The previous discussion completes the proof of Theorem \ref{EX-PosNegYamabe}, part (2).

\subsection{WPE Asymptotically Minimal Submanifolds - Proof of Theorem \ref{WPE-Submanifolds}}
Before proceeding, we define a convenient coordinate system on our hypersurface. Let \( Y^{n+1} \) be a conformally compact hypersurface of an AH manifold \( (M^{n+2}, g_+) \). Let \( r \) be a special defining function for \( \Sigma \). As usual, we set $\bar g = r^2g_+$ and $\hat{g} = (r^2 g_+) |_{T^2\Sigma}$. Let \( x^1, \dots, x^n \) be local coordinates in an open set $U$ in \( \partial Y \), and let \( x^{n+1} \) denote a signed \( \hat{g} \)-distance to \( \partial Y \). Extend $x^1,...,x^{n}$ to a neighborhood of $U$ in $\Sigma$ using the flow of $\nabla_{\hat{g}}x^{n+1}.$ Then extend $x^1,...,x^n,x^{n+1}$ to an on open set in $M$ by the flow of $\nabla_{r^2g_+}r.$ Let \( 1 \leq a,b,c,d \leq n \), and let \( \alpha, \beta, \gamma, \delta \in \{r,1,2,...,n \} \) . Furthermore, let \( \hat{x}^{\alpha} = x^{\alpha} \big|_{Y} \) and $\hat{r}=r|_{Y}$. We refer to these as induced holographic coordinates on $Y.$ We let the index $0$ correspond to $r$ when a tensor is evaluated at $r=0.$ Let \( H^Y \) and \( \overline{H}^Y \) denote the scalar mean curvatures of \( Y \) with respect to \( g_+ \) and \( \overline{g} \), respectively. Similarly, let \( N \) and \( \overline{N} \) be unit normal vectors to \( Y \) with respect to \( g_+ \) and \( \overline{g} \), respectively, oriented in the same direction. Notice that $N$ and $\overline{N}$ are conformally related by $N=r\overline{N}.$ 

\begin{figure}[h]\label{MinHyp}
\resizebox{\textwidth}{!}{%
\begingroup
\definecolor{diagramblue}{RGB}{8,119,242}

\begin{tikzpicture}[
  x=1cm,
  y=1cm,
  line cap=round,
  line join=round,
  >={Stealth[length=3.0mm,width=2.2mm]},
  blackline/.style={black,line width=1.7pt},
  blueline/.style={diagramblue,line width=2.2pt},
  sametip/.style={-{Stealth[length=3.0mm,width=2.2mm]}},
  rvector/.style={black,line width=1.7pt,sametip},
  hatvector/.style={diagramblue,line width=1.9pt,sametip},
  hidden/.style={dash pattern=on 5mm off 3mm},
  every node/.style={inner sep=1pt}
]

\coordinate (Ptl) at (1.65,4.52);
\coordinate (Ptr) at (16.10,4.52);
\coordinate (Pbr) at (14.82,0.15);
\coordinate (Pbl) at (0.18,0.15);

\coordinate (L) at (3.31,2.12);
\coordinate (R) at (12.04,1.72);

\draw[blackline] (Ptl)--(4.03,4.52);
\draw[blackline,hidden] (4.03,4.52)--(10.72,4.52);
\draw[blackline]
  (10.72,4.52)--(Ptr)--(Pbr)--(Pbl)--(Ptl);

\draw[blueline,hidden]
  (L)
  .. controls (4.48,3.04) and (8.26,3.14) ..
  (R);

\draw[blueline]
  (L)
  .. controls (3.31,3.20) and (3.93,7.28) ..
  (5.15,9.04)
  .. controls (5.82,10.02) and (6.70,9.63) ..
  (7.52,8.91)
  .. controls (9.07,7.51) and (10.13,5.28) ..
  (10.72,4.52)
  .. controls (11.05,4.10) and (12.04,2.72) ..
  (R);

\coordinate (FoldTop) at (5.89,7.36);
\coordinate (FoldBottom) at (5.55,4.57);

\draw[blueline]
  (6.63,8.23)
  .. controls (6.35,7.98) and (6.10,7.65) ..
  (FoldTop)
  .. controls (5.58,6.75) and (5.25,5.35) ..
  (FoldBottom)
  .. controls (5.60,4.30) and (5.71,3.98) ..
  (5.84,3.73);

\draw[blueline]
  (FoldTop)
  .. controls (6.04,6.35) and (5.93,5.20) ..
  (FoldBottom);

\draw[blueline]
  (L)
  .. controls (5.20,1.27) and (10.25,0.73) ..
  (R);

\coordinate (Ls) at (3.31,3.20);
\coordinate (Lb) at (4.31,2.12);

\pic[
  draw=diagramblue,
  line width=1.7pt,
  angle radius=4.4mm
] {right angle=Ls--L--Lb};

\coordinate (Rs) at (12.04,2.72);
\coordinate (Rb) at (11.04,1.72);

\pic[
  draw=black,
  line width=1.6pt,
  angle radius=4.4mm
] {right angle=Rs--R--Rb};

\coordinate (NBone) at (3.400,3.212);
\coordinate (NTone) at (3.444,3.559);
\coordinate (NEone) at (2.507,3.325);

\draw[blackline,sametip]
  (NBone)--(NEone);

\pic[
  draw=black,
  line width=1.4pt,
  angle radius=2.7mm
] {right angle=NTone--NBone--NEone};

\coordinate (NBtwo) at (3.352,2.791);
\coordinate (NTtwo) at (3.387,3.139);
\coordinate (NEtwo) at (2.456,2.880);

\draw[blackline,sametip]
  (NBtwo)--(NEtwo);

\pic[
  draw=black,
  line width=1.4pt,
  angle radius=2.7mm
] {right angle=NTtwo--NBtwo--NEtwo};

\path[yshift=4.5mm]
  (L)
  .. controls (5.20,1.27) and (10.25,0.73) ..
  (R)
  node[pos=.3669606,coordinate]
  (BlueGuideBase) {};

\path[yshift=3.8mm]
  (L)
  .. controls (5.20,1.27) and (5,0.6) ..
  (R)
  node[pos=.2586859,coordinate]
  (BlackGuideBase) {};

\draw[blackline,sametip]
  (BlackGuideBase)--++(-0.6,-0.9);

\node[font=\Large] at (4.66,0.51)
  {$\partial_{n+1}$};

\coordinate (Qzero) at (12.04,1.72);
\coordinate (Qone) at (11.843,2.620);
\coordinate (Qtwo) at (11.274,3.727);
\coordinate (QoneHat) at (11.843,2.620);
\coordinate (QtwoHat) at (11.274,3.727);

\draw[rvector]
  (Qzero)--++(0,0.86);

\draw[hatvector]
  (QoneHat)--++(-0.225,0.72);

\draw[rvector]
  (Qone)--++(0,0.80);

\node[
  font=\Large,
  anchor=west
] at (12.02,3.02)
  {$\partial_r$};

\node[
  text=diagramblue,
  font=\Large,
  anchor=east
] at (11.32,2.97)
  {$\partial_{\hat r}$};

\draw[hatvector]
  (QtwoHat)--++(-0.345,0.780);

\draw[rvector]
  (Qtwo)--++(0,0.79);

\node[
  font=\Large,
  anchor=west
] at (11.46,4.11)
  {$\partial_r$};

\node[
  text=diagramblue,
  font=\Large,
  anchor=east
] at (10.77,4.06)
  {$\partial_{\hat r}$};

\node[font=\huge] at (2.65,7.94)
  {$M^{n+2}$};

\node[font=\huge] at (7.50,7.08)
  {$Y^{n+1}$};

\node[font=\huge] at (1.86,0.78)
  {$\Sigma^{n+1}$};

\node[font=\Large] at (2.95,3.92)
  {$\overline N$};

\node[font=\Large] at (9.96,1.55)
  {$\partial Y$};

\node[
  anchor=west,
  font=\Large
] at (12.19,1.48)
  {\textcolor{diagramblue}{$\partial_{\hat r}$}
   {$\,=\,\partial_r$}};

\end{tikzpicture}
\endgroup%
}
    \caption{An asymptotically minimal hypersurface $Y^{n+1}$ sitting inside an AH manifold $M^{n+2}$.}
    \label{fig:enter-label}
\end{figure}
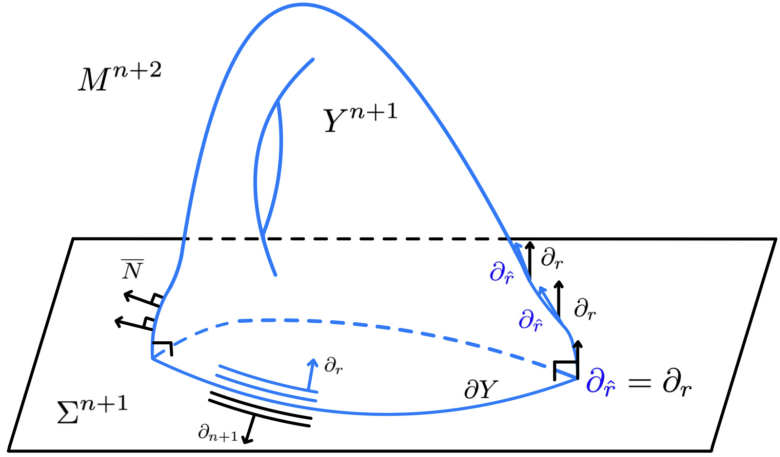

We use \( h_+ \) and \( \overline{h} \) to denote the induced metrics on \( Y \). Specifically, we define  
\[
h_{\alpha\beta} = g_+(\partial_{\hat{\alpha}}, \partial_{\hat{\beta}}),\quad \overline{h}_{\alpha\beta} = \overline{g}(\partial_{\hat{\alpha}}, \partial_{\hat{\beta}}),\quad B_{\alpha\beta} = B(\partial_{\hat{\alpha}}, \partial_{\hat{\beta}}) = g_+(\nabla^{g_+}_{\hat{\alpha}} \partial_{\hat{\beta}}, N),\]
and
\[
\overline{B}_{\alpha\beta} = \overline{B}(\partial_{\hat{\alpha}}, \partial_{\hat{\beta}}) = \overline{g}(\nabla^{\overline{g}}_{\hat{\alpha}} \partial_{\hat{\beta}}, \overline{N}).
\]
We will use $II$ to denote the second fundamental form of $\partial Y$ in $\Sigma.$

Let us assume now the bulk manifold $(M^{n+2},g_+)$ is PE. The conformal transformation law for the second fundamental form states
\[
B_{\alpha\beta}=\frac{\overline{B}_{\alpha\beta}}{r}+\frac{\overline{N}(r)\overline{h}_{\alpha\beta}}{r^2}.
\]
Contracting yields
\[
\overline{N}(r)=\frac{H^Y-r\overline{H}^Y}{n+1},
\]
thus it follows that $H^Y=O(r)$ if and only if $\overline{N}(r)=0$ at $r=0.$ Since $\overline{N}(r)=\overline{g}(\partial_r,\overline{N})$, we observe that $H^Y=O(r)$ is equivalent to $Y$ meeting $\Sigma$ at right angles. 

We will use the following two formulas, which we carefully derived in the Appendix \ref{SubmanifoldGeometry}:
 \begin{equation}\label{RicciSubmanifod}
    R_{\alpha\gamma}^Y+nh_{\alpha\gamma}=-B^2_{\alpha\gamma}+H^YB_{\alpha\gamma}-W^{M}_{\hat{\alpha} N\hat{\gamma} N},
    \end{equation}
and
\begin{equation}\label{SCBABY}
    R^Y+n(n+1)=(H^Y)^2-|B|^2_{h_+}.
\end{equation} Our goal is to apply Theorem \ref{WPE-Characterization} to $(Y^{n+1},h_+)$ based on the previous expressions. For instance, from the above expansions and \[ W^{M}_{\hat{\alpha} N\hat{\gamma} N} = W^M(\partial_{\hat{\alpha}}, N,\partial_{\hat{\beta}}, N) = r^2W^M(\partial_{\hat{\alpha}},\overline N, \partial_{\hat{\beta}},\overline N) = \overline W_{\hat{\alpha} \overline N\hat{\gamma} \overline N},\] we obtain that under the assumption that $H^Y=O(r)$ it necessarily follows that $\partial Y$ is totally geodesic in $(\overline{Y},\overline{h})$.

Let $X,Z$ be vectors tangent to $\partial Y$ at a point $q\in \partial Y$. Recall that since $Y$ is asymptotically minimal, it meets $\Sigma$ at right angles. Therefore, $\overline{N}$ is tangent to $\Sigma$ at points on $\partial Y$, and so at $q$ we have
  \[
  \overline{B}(X,Z)=\overline{g}(\nabla_{X}^{\overline{g}}Z,\overline{N}),\quad
   II(X,Z)=\hat{g}(\nabla_{X}^{\hat{g}}Z,\overline{N}).
  \]
Recall that $\nabla^{\hat{g}}_XZ=\mathrm{Proj}(\nabla^{\overline{g}}_XZ)$, where Proj is the projection onto $T\Sigma.$ It follows that 
\[
II(X,Z)=\overline{g}(\mathrm{Proj}(\nabla^{\overline{g}}_XZ),\overline{N})=\overline{g}(\nabla^{\overline{g}}_XZ,\overline{N})=\overline{B}(X,Z).
\]
We sum this up with a lemma:
\begin{lemma}\label{2ffsEqual}
If $Y$ is an asymptotically minimal hypersurface in a PE space\footnote{This is still true if PE is weakened to AH.}, then 
\begin{equation}
II=\overline{B}|_{T^2\partial Y}.
\end{equation}
\end{lemma}
With the above lemma in mind, along with $\overline{h}^{ar}=O(r^2)$, we deduce that $\overline{H}^Y=\hat{\eta}+\overline{B}_{rr}+O(r^2),$ where $\hat{\eta}=\hat{h}^{ab}II_{ab}$ and $\hat{h}=\overline{h}|_{T^2\partial Y}.$ 

\vspace{.1in}
We are now in position to prove \cref{conf inv}: 
\begin{proof}[Proof of Proposition \ref{conf inv}]
Recall from (\ref{RicciSubmanifod}) that we have  
\[
 R_{rr}^Y+nh_{rr}=-B^2_{rr}+H^YB_{rr}-W^{M}_{\hat{r}N\hat{r}N},
 \]
where $B^2_{rr}=h^{\alpha\beta}B_{r\alpha}B_{\beta r}=r^2\overline{h}^{rr}(B_{rr})^2+r^2\overline{h}^{ab}B_{ra}B_{br}+2r^2\overline{h}^{ar}B_{ar}B_{rr}$. By \cref{RicciCurvatureExpansion1}, we know that \(R_{rr}^Y+nh_{rr}=O(r)\) since \(\partial Y\) is totally geodesic in \(\overline{Y}\). Next, since $ W^{M}_{\hat{\alpha} N\hat{\gamma} N} = \overline W^M_{\hat{\alpha} \overline N\hat{\gamma} \overline N}$ and the third conformal fundamental form  of $\Sigma$ vanishes, as we are assuming the bulk is PE, it is clear that \(W^{M}_{\hat{r}N\hat{r}N}=O(r)\). Thus it follows that  
\begin{equation}\label{conf inv2}
-B^2_{rr}+H^YB_{rr}=O(r).
\end{equation}

We proceed by noting that 
\[
B_{ar}=\frac{\overline{B}_{ar}}{r}+\frac{\overline{N}(r)\overline{h}_{ar}}{r^2},
\]
but recalling that $\overline{N}|_{r=0}=\partial_{n+1}$, we deduce 
\[
\begin{split}
\overline{B}_{a0}&=\overline{g}(\nabla^{\overline{g}}_{\hat{a}}\partial_{\hat{r}},\partial_{n+1})|_{r=0}=\overline{\Gamma}^{n+1}_{\hat{a}\hat{r}}|_{r=0}\\
&=\frac{1}{2}\{\partial_{\hat{a}}\overline{g}_{\hat{r} n+1}+\partial_{\hat{r}}\overline{g}_{\hat{a} n+1}-\partial_{n+1}\overline{g}_{\hat{a}\hat{r}}\}|_{r=0}=0.
\end{split}
\]
This allows us to write $B^2_{rr}=r^2(B_{rr})^2+O(r^2)$ and,
consequently, (\ref{conf inv2}) becomes
\begin{equation}\label{conf inv3}
-r^2(B_{rr})^2+H^YB_{rr}=O(r).
\end{equation}

Recall once again that  
\[
B_{\alpha\beta}=\frac{\overline{B}_{\alpha\beta}}{r}+\frac{\overline{N}(r)\overline{h}_{\alpha\beta}}{r^2},
\]
and  
\[
\overline{N}(r)=\frac{H^Y-r\overline{H}^Y}{n+1}.
\]
From this it follows that  
\begin{equation}\label{dN}
\partial_r|_{r=0}(\overline{N}(r))=\frac{\partial_r|_{r=0}(H^Y)-\overline{H}^Y|_{r=0}}{n+1}=\frac{\partial_r|_{r=0}(H^Y)-\hat{\eta}-\overline{B}_{00}}{n+1}.
\end{equation}
After checking that $\overline{h}_{rr}=1+O(r^2)$ and applying the conformal transformation law for \(B\) and using (\ref{conf inv3}), we obtain  
\[
-\left(\overline{B}_{rr}+\frac{\overline{N}(r)}{r}\right)^2+\frac{H^Y}{r}\left(\overline{B}_{rr}+\frac{\overline{N}(r)}{r}\right)=O(r),
\]
thus
\[
\left(\frac{H^Y}{r}-\overline{B}_{rr}-\frac{\overline{N}(r)}{r}\right)\left(\overline{B}_{rr}+\frac{\overline{N}(r)}{r}\right)=O(r).
\]
Taking the limit as \(r \to 0\), we deduce that either 
\begin{enumerate}
\item
$\partial_{r}|_{r=0}(H^Y)=\overline{B}_{00}+\partial_{r}|_{r=0}(\overline{N}(r))$
\\or 
\item
$\partial_{r}|_{r=0}(\overline{N}(r))=-\overline{B}_{00}.$
\end{enumerate}

After using (\ref{dN}), we obtain that if $(1)$ holds then
\[
\partial_{r}|_{r=0}(H^Y)=\overline{B}_{00}-\frac{\hat{\eta}}{n},
\]
while if $(2)$,  
\[
\partial_{r}|_{r=0}(H^Y)=\hat{\eta}-n\overline{B}_{00}.
\]
In particular,  
\[
H^Y=C\bigg(\hat{\eta}-n\overline{B}_{00}\bigg)r+O(r^2),
\]
for some $C\in \{1, -\frac{1}{n} \}$. This implies \(H^Y=O(r^2)\) if and only if $\hat{\eta}-n\overline{B}_{00}=0.$

Note that if $(1)$ or $(2)$ is true for some sdf $r$ then it is also true for any choice of sdf $r$ since $H^Y=h_{+}^{\alpha\beta}g_{+}(\nabla^{g_+}_{\hat{\alpha}}\partial_{\hat{\beta}}, N)$ and so expanding this in terms of an sdf one obtains a universal formula for the first derivative of $H^Y$ in terms of extrinsic Riemaniann invariants expressed in terms of the pullback of the metric $r^2g_+$ to $\overline{Y}.$  

Therefore after expanding \(H^Y\) with respect to a different choice of special defining function, $\rho$ corresponding to \(\tilde{g}=e^{2w}\hat{g}\), and letting $\tilde{\eta}, \tilde{\overline{B}}_{00}$ be the the corresponding quantities for $\tilde{\overline{g}}=\rho^2g_+$, we obtain  
\begin{equation}\label{NewConfInv}
\hat{\eta}-n\overline{B}_{00}=e^{w}\big(\tilde{\eta}-n\tilde{\overline{B}}_{00}\big),
\end{equation}
proving conformal invariance. This concludes the proof.
\end{proof} 

We now turn our attention to \cref{WPE-Submanifolds}. Before proceeding with the proof, we need the following lemma:

\begin{lemma}\label{Weyl Lemma}
    Let $(M,g_+)$ be a PE space with $\Sigma=\partial M$. Let $r$ be a special defining function for $\Sigma$ and let $i,j,k,l$ be indices corresponding to the coordinates of a holographic coordinate system induced by $r$ which are different from $r.$ Then 
    \begin{equation}
\overline{W}^M_{ijkl}=W^{\Sigma}_{ijkl}+O_{ijkl}(r).
    \end{equation}
 By $W^{\Sigma}_{ijkl}$ we mean the Weyl tensor on $\Sigma$ corresponding to $(r^2g_+)|_{T^2\Sigma}.$  
\end{lemma}
\begin{proof}

It is straightforward to check that the Einstein condition implies 
\begin{equation}\label{Weyl}
W_{ijkl}^{M}=R^{M}_{ijkl}+(g_{ik}g_{jl}-g_{il}g_{jk})
\end{equation}
(note we are using the shorthand $g_{ij}=(g_+)_{ij}$). The conformal transformation law for the Riemann tensor is given by 
\begin{equation*}
R^{M}_{ijkl}=r^{-2}\left[\overline{R}^M_{ijkl}+(\overline{g}_{il}\overline{T}_{jk}+\overline{g}_{jk}\overline{T}_{il}-\overline{g}_{ik}\overline{T}_{jl}-\overline{g}_{jl}\overline{T}_{ik})\right]
\end{equation*}
where
\begin{equation*}
    \overline{T}_{ij}=-\nabla^{\overline{g}}_i\nabla^{\overline{g}}_j \log r-\nabla^{\overline{g}}_i \log r\nabla^{\overline{g}}_j \log r+\frac{1}{2}|d\log r|^2
\overline{g}_{ij}. 
\end{equation*}
Next we compute the three terms in the right-hand-side of the formula above: 
\begin{itemize}
    \item $\nabla^{\overline{g}}_i \log r= \frac{r_i}{r}.$
    \item $\nabla^{\overline{g}}_i\nabla^{\overline{g}}_j \log r=\frac{r_{ij}}{r}-\frac{r_ir_j}{r^2}-\overline{\Gamma}_{ij}^k\frac{r_k}{r}-\overline{\Gamma}^r_{ij}\frac{1}{r}.$
    \item $|d \log r|^2_{\overline{g}}=\overline{g}^{ij}\frac{r_ir_j}{r^2}.$
\end{itemize}
Note that $r_i=0$ and $\overline{\Gamma}_{ij}^r=-\frac{\overline{g}_{ij}'}{2},$ which allows us to write
\begin{equation*}
\overline{T}_{ij}=\frac{\overline{g}_{ij}}{2r^2}-\frac{\overline{g}_{ij}'}{2r}.
\end{equation*}
Using this and the above formulas we then arrive at
\begin{align*}\label{Riemg_+}
R^{M}_{ijkl}&=r^{-2}\left[\overline{R}^M_{ijkl}+\frac{\overline{g}_{ik}\overline{g}_{jl}'-\overline{g}_{il}\overline{g}_{jk}'+\overline{g}_{ik}'\overline{g}_{jl}-\overline{g}_{il}'\overline{g}_{jk}}{2r}+\frac{\overline{g}_{jk}\overline{g}_{il}-\overline{g}_{ik}\overline{g}_{jl}}{r^2} \right] \\\nonumber
&= (g_{il}g_{jk}-g_{ik}g_{jl})+r^{-2}\left[\overline{R}^{M}_{ijkl}+\frac{\overline{g}_{ik}\overline{g}_{jl}'-\overline{g}_{il}\overline{g}_{jk}'+\overline{g}_{ik}'\overline{g}_{jl}-\overline{g}_{il}'\overline{g}_{jk}}{2r}\right].
\end{align*}
By \eqref{Weyl}, it now follows that 
\begin{equation*}\label{W}
W^{M}_{ijkl}=r^{-2}\left[\overline{R}^{M}_{ijkl}+\frac{\overline{g}_{ik}\overline{g}_{jl}'-\overline{g}_{il}\overline{g}_{jk}'+\overline{g}_{ik}'\overline{g}_{jl}-\overline{g}_{il}'\overline{g}_{jk}}{2r}\right].
\end{equation*}
Next recall that 
\begin{equation}\label{g expand}
\overline{g}_{ij}=\hat{g}_{ij}-\hat{P}_{ij}r^2+O(r^3),
\end{equation}
After recalling that $\overline{W}^{M}_{ijkl}=r^2W^{M}_{ijkl},$ and the fact that $\Sigma$ is totally geodesic with respect to $\overline{g},$ we may write \cref{Weyl}
\begin{equation*}\label{W'}
\begin{split}
\overline{W}^{M}_{ijkl}&=\overline{R}^{M}_{ijkl}+\frac{\overline{g}_{ik}\overline{g}_{jl}'-\overline{g}_{il}\overline{g}_{jk}'+\overline{g}_{ik}'\overline{g}_{jl}-\overline{g}_{il}'\overline{g}_{jk}}{2r}\\
&=R^{\Sigma}_{ijkl}-\hat{g}_{ik}\hat{P}_{jl}+\hat{g}_{il}\hat{P}_{jk}-\hat{P}_{ik}\hat{g}_{jl}+\hat{P}_{il}\hat{g}_{jk}+O_{ijkl}(r)\\
&=W^{\Sigma}_{ijkl}+O_{ijkl}(r).
\end{split}
\end{equation*}

\end{proof}

Finally, we prove \cref{WPE-Submanifolds}:

\begin{proof}[Proof of Theorem \ref{WPE-Submanifolds}]
  By assumption $H^Y=O(r^2)$, thus $\overline{B}_{00}=\frac{\hat{\eta}}{n}$ by \cref{conf inv}. It follows that 
at points on $\partial Y$
\[
\frac{\overline{H}^Y}{n+1}=\frac{\hat{\eta}}{n}. 
\]
Therefore along $\partial Y$ it holds that
\[
\mathring{\overline{B}}_{ab}=\overline{B}_{ab}-\frac{\overline{H}^Y}{n+1}\hat{g}_{ab}=II_{ab}-\frac{\hat{\eta}}{n}\hat{g}_{ab}=\mathring{II}_{ab},
\]
where we have used Lemma \ref{2ffsEqual}.Now, recall the law $\mathring{B}_{\alpha\beta}=\dfrac{\mathring{\overline{B}}_{\alpha\beta}}{r}$ and observe that by $H^Y=O(r^2)$, we get $B_{\alpha\beta}=\mathring{B}_{\alpha\beta}+O(r^2)g_{\hat{\alpha}\hat{\beta}}$. Therefore,
\begin{equation}\label{|B|}
\begin{split}
|B|^2_{h_+}=|\mathring{II}|^2r^2+O(r^3). 
\end{split}
\end{equation}
Moreover, after writing $\overline{N}=\partial_{n+1}+O(r), $ letting the index $\overline{n}$ correspond to $\partial_{n+1}|_{r=0}$ and using \cref{Weyl Lemma}, we may write $W^{M}_{\hat{\alpha}\overline{N}\hat{\beta}\overline{N}}=r^{-2}\overline{W}^M_{\hat{\alpha}\overline{N}\hat{\beta}\overline{N}}=r^{-2}(W^{\Sigma}_{a\overline{n}b\overline{n}}\delta_{\alpha}^{a}\delta_{\beta}^{b}+O_{\alpha\beta}(r))$. Furthermore, by equations (\ref{RicciSubmanifod}), (\ref{SCBABY}) we obtain

\begin{align}\label{Hypersurface Curvatures}
R^{Y}_{\alpha\beta}+nh_{\alpha\beta}&=-B^2_{\alpha\beta}-r^2W^{M}_{\hat{\alpha} \overline{N}\hat{\beta} \overline{N}} + O_{\alpha\beta}(r)\\\nonumber
&=-\mathring{II}^2_{ab}\delta_{\alpha}^a \delta_{\beta}^{b}-W^{\Sigma}_{a \overline{n}b \overline{n}}\delta_{\alpha}^a \delta_{\beta}^{b}+O_{\alpha\beta}(r),\\
R^{Y}+n(n+1)&=-|\mathring{II}|^2r^2+O(r^3).
\end{align}
The result now follows from \cref{WPE-Characterization}.
\end{proof}

\begin{remark}
In \cref{Weyl Lemma} if we have that $(\overline{M},\overline{g})$ has $\mathrm{dim}(M)>4$ of if $\mathrm{dim}(M)=4$ but $g_{(3)}\equiv 0$ (i.e. $(M,g_+)$ is critical for renormalized volume) then we can replace the $O(r^3)$ in \cref{g expand} with $O(r^4)$ and get \begin{equation}
\overline{W}^M_{ijkl}=W^{\Sigma}_{ijkl}+O_{ijkl}(r^2).
\end{equation}
This is stronger than what we need for our result.
\end{remark}
\begin{remark}
    If we extend the coordinate system $\{x^{\hat{\alpha}}\}$ from a neighbourhood $U_0 $ of $\partial \overline{Y}$ in $\overline{Y}$ to coordinates $\{x^{\hat{\alpha}}, \nu \}$ on a neighbourhood $U$ in $M$ by letting $\nu$ be a signed $\overline{g}-$distance function to $\overline{Y}$ and extending the coordinates off the hypersurface via the flow of $\nabla^{\overline{g}} \nu$ then we may write $\overline{B}_{\alpha\beta}=\overline{g}(\overline{\nabla}_{\hat{\alpha}}\partial_{\hat{\beta}}, \partial_{\nu})=-\frac{1}{2}\partial_{\nu}\overline{g}_{\hat{\alpha}\hat{\beta}}.$ Then since $\partial_{\hat{r}}=\partial_{r}$ along $\Sigma$ and $\partial_{\nu}$ is tangent to $\Sigma$ along $\Sigma$ we get $\partial_{\nu}(\partial_{r}-\partial_{\hat{r}})=0.$ Then after recalling that $\partial_r|_{r=0} \overline{g}_{\hat{\alpha}\hat{\beta}}=0$   
    it follows that 
    $\overline{B}_{\alpha\beta}=II_{ab}\delta^a_{\alpha}\delta^{b}_{\beta}+\overline{B}_{00}\delta^r_{\alpha}\delta^{r}_{\beta}+O_{\alpha\beta}(r^2).$ Which, after using $H^Y=O(r^2)$ and $\cref{conf inv},$ in turn implies $\mathring{\overline{B}}_{\alpha\beta}=\mathring{II}_{ab}\delta^a_{\alpha}\delta^{b}_{\beta}+O_{\alpha\beta}(r^2).$ From which it follows that \eqref{|B|} can be improved to give

    \begin{equation}
    |B|^2_{h_+}=|\mathring{II}|^2r^2+O(r^4).
    \end{equation}
    This is more than what is required for our purposes.
\end{remark}

\section{Appendix}\label{SubmanifoldGeometry}

\subsection{Ricci and Scalar curvature of hypersurfaces $Y^{n+1}$ in a PE space.}
Let $(M^{n+2},g_+)$ be a PE space. We will use Greek letters to denote tangent directions to $Y$, and we will use $N$ to denote the index corresponding to a coordinate function given by a choice of signed $g_+$-distance to $Y.$ 
We let $h_+$ be the induced metric on $Y.$ First note that since $g_+$ is PE we may write 
\begin{equation}
R^M_{\alpha\beta\gamma\delta}=-(g_{\alpha\gamma}g_{\beta\delta}-g_{\alpha\delta}g_{\beta\gamma})+W^M_{\alpha\beta\gamma\delta}.
\end{equation}
By Gauss' equation, we have 
    \[
    R^{M}_{\alpha\beta\gamma\delta}=R^Y_{\alpha\beta\gamma\delta}+B_{\alpha\delta}B_{\beta\gamma}-B_{\alpha\gamma}B_{\beta\delta},
    \]
    where $B$ is the (scalar) second fundamental form of $Y^{n+1}$.
    Tracing the above equation with $g$ gives
    \[
    g^{\beta\delta}R^M_{\alpha\beta\gamma\delta} = R^Y_{\alpha\gamma} + B^2_{\alpha\gamma}-HB_{\alpha\gamma}.
    \]
    On the other hand,
    \[
    g^{\beta\delta}R^M_{\alpha\beta\gamma\delta} = R^{M}_{\alpha\gamma} - R^{M}_{\alpha N\gamma N} = -ng_{\alpha\gamma}-W^{M}_{\alpha N \gamma N},
    \]
    thus
    \begin{equation}\label{RicciSubmanifod2}
    R_{\alpha\gamma}^Y+ng_{\alpha\gamma}=-B^2_{\alpha\gamma}+HB_{\alpha\gamma}-W^{M}_{\alpha N\gamma N}.
    \end{equation}
After taking another trace, we see
\begin{equation}\label{SCBABY2}
    R^Y+n(n+1)=H^2-|B|^2_{h_+}
\end{equation}
where $h_+$ is the induced singular metric on $Y.$

\subsection{Fialkow--Gauss Equation}

We work with a submanifold $Y^{k+1}$ of a manifold $(M^{n+1},g)$, where $k\ge 2$. Throughout this section, $B$ and $H$ will denote the second fundamental form and the mean curvature vector of $(Y^{k+1}, h)$, respectively. We will use $m=n-k$ to denote the codimension, and we will use $\eta',\mu',\cdots$ to denote normal directions and $i,j,\cdots$ to denote tangential directions in an adapted orthonormal frame\footnote{The letter ``k" will appear in two ways, as the dimension of the boundary of the submanifold and as an index representing a tangential direction.}. Our goal is to find a relation between $P^M$ and $P^Y$, that is, a relation between the extrinsic and intrinsic Schouten tensors.

We start with the Gauss equation and write
\begin{equation}\label{GaussEquation-1}
R^M_{ijkl}=R^Y_{ijkl}+\sum_{\eta'=1}^m( B_{il}^{\eta'} B_{jk}^{\eta'}- B_{ik}^{\eta'} B_{jl}^{\eta'}).
\end{equation}
Tracing the equation with $h^{jl}$ and rearranging gives
\begin{equation}\label{GaussEquationTraced1}
R^Y_{ik} =  R^M_{ik} -  g^{\mu'\nu'}R^M_{i\mu'k\nu'}-\sum_{\eta'=1}^m ( B^{\eta'})^2_{ik} + \sum_{\eta'=1}^m B^{\eta'}_{ik} H^{\eta'}
\end{equation}
Equation \ref{GaussEquationTraced1} gives an expression for the Ricci curvature of $Y$ in terms of extrinsic geometry. Tracing (\ref{GaussEquationTraced1})  with $h^{ik}$ gives an expression for the scalar curvature of the submanifold $Y$:
\begin{equation}\label{GaussEquationTraced2}
\begin{split}
    R^Y &= R^M - \tensor{{R^g_+}}{_{\mu'}^{\mu'}} - h^{ik}g^{\mu'\nu'}R^M_{i\mu' k\nu'} - \sum_{\eta'=1}^m |B^{\eta'}|^2 + \sum_{\eta'=1}^m (H^{\eta'})^2 \\ &= R^M - \tensor{{R^g_+}}{_{\mu'}^{\mu'}} - h^{ik} g^{\mu'\nu'}R^M_{i\mu'k\nu'} - |B|_{g}^2 + |H|_{g}^2
\end{split}
\end{equation}
On the other hand, the Riemann curvature tensor can be decomposed into its Weyl and Schouten parts as follows
\begin{equation}\label{CurvatureDecomposition}
    R^M_{i\mu'k\nu'} = W^M_{i\mu'k\nu'} + P^M_{ik} g_{\mu'\nu'} + P^M_{\mu'\nu'}h_{ik},
\end{equation}
where the extra two terms vanish as it involves the inner product of tangential with normal directions. Tracing in (\ref{CurvatureDecomposition}) with $g^{\mu'\nu'}$ yields
\begin{equation}
\begin{split}
    g^{\mu'\nu'}R^M_{i\mu'k\nu'} & = \tensor{{W^M}}{_i_{\mu'}_k^{\mu'}} + (n-k)P^M_{ik} +  g^{\mu'\nu'}P^M_{\mu'\nu'}h_{ik} \\ &= \tensor{{W^M}}{_i_{\mu'}_k^{\mu'}} + \frac{n-k}{n-1}R^M_{ik} - \frac{n-k}{n(n-1)}R^Mh_{ik} +\frac{1}{n-1}\tensor{{R^M}}{_{\mu'}^{\mu'}}h_{ik},
\end{split}
\end{equation}
and
\begin{equation}
\begin{split}
    h^{ik} g^{\mu'\nu'}R^M_{i\mu'k\nu'} &= \tensor{{W^M}}{_a_{\mu'}^a^{\mu'}} +\frac{n-k}{n-1}(R^M - \tensor{{R^M}}{_{\mu'}^{\mu'}}) - \frac{(k+1)(n-k)}{n(n-1)}R^M + \frac{k+1}{n-1}\tensor{{R^M}}{_{\mu'}^{\mu'}} \\&= \tensor{{W^M}}{_a_{\mu'}^a^{\mu'}}+ \frac{(n-k)(n-k-1)}{n(n-1)}R^M + \frac{2k-n+1}{n-1}\tensor{{R^M}}{_{\mu'}^{\mu'}}
\end{split}
\end{equation}

We are now in position to study the Schouten tensors. Using (\ref{GaussEquationTraced1}) and (\ref{GaussEquationTraced2}), we deduce
\[
\begin{split}
    P^Y_{ik} =&\; \frac{1}{k-1}R^Y_{ik} - \frac{R^Y}{2k(k-1)}h_{ik} \\ =&\; \frac{1}{k-1}\left(R^M_{ik} -  g^{\mu'\nu'}R^M_{i\mu'k\nu'}-\sum_{\eta'=1}^m ( B^{\eta'})^2_{ik} + \sum_{\eta'=1}^m  B^{\eta'}_{ik} H^{\eta'}\right) \\& - \frac{h_{ik}}{2k(k-1)}\left(R^M - \tensor{{R^M}}{_{\mu'}^{\mu'}} - h^{ab} g^{\mu'\nu'}R^M_{a\mu'b\nu'} - |B|^2 + |H|^2\right) \\ =&\;\frac{1}{k-1}\left( \sum_{\eta'=1}^m  B^{\eta'}_{ik} H^{\eta'} - \sum_{\eta'=1}^m ( B^{\eta'})^2_{ik} \right) + \frac{1}{k-1}R^M_{ik} \\&-\frac{1}{k-1}\left(\tensor{{W^M}}{_i_{\mu'}_k^{\mu'}} + \frac{n-k}{n-1}R^M_{ik} - \frac{n-k}{n(n-1)}R^Mh_{ik} +\frac{1}{n-1}\tensor{{R^M}}{_{\mu'}^{\mu'}}h_{ik}\right) \\ &-\frac{h_{ik}}{2k(k-1)}\left(|H|^2 - |B|^2\right) - \frac{h_{ik}}{2k(k-1)} R^M + \frac{h_{ik}}{2k(k-1)}\tensor{{R^M}}{_{\mu'}^{\mu'}} \\ &+\frac{h_{ik}}{2k(k-1)}\left(\tensor{{W^M}}{_{a}_{\mu'}^{a}^{\mu'}}+ \frac{(n-k)(n-k-1)}{n(n-1)}R^M + \frac{2k-n+1}{n-1}\tensor{{R^M}}{_{\mu'}^{\mu'}}\right)
\end{split}
\]
Therefore,
\[
\begin{split}
    P^Y_{ik} =& \frac{1}{k-1}\left( \sum_{\eta'=1}^m  B^{\eta'}_{ik} H^{\eta'} - \sum_{\eta'=1}^m ( B^{\eta'})^2_{ik}-\tensor{{W^M}}{_i_{\mu'}_k^{\mu'}} \right)-\frac{h_{ik}}{2k(k-1)}\left(|H|^2 - |B|^2\right) \\& + \tensor{{R^M}}{_{\mu'}^{\mu'}}h_{ik}\left(\frac{1}{2k(k-1)} - \frac{1}{(n-1)(k-1)} + \frac{2k-n+1}{2(n-1)k(k-1)}\right) \\ &+ R^Mh_{ik}\left(\frac{n-k}{(k-1)n(n-1)} - \frac{1}{2k(k-1)} + \frac{(n-k)(n-k-1)}{2k(k-1)n(n-1)}\right) \\ &+ R^M_{ik}\left(\frac{1}{k-1}-\frac{n-k}{(k-1)(n-1)}\right) + \frac{h_{ik}}{2k(k-1)} \tensor{{W^M}}{_{a}_{\mu'}^{a}^{\mu'}} \\ =& \; \frac{1}{k-1}\left( \sum_{\eta'=1}^m  B^{\eta'}_{ik}H^{\eta'} - \sum_{\eta'=1}^m ( B^{\eta'})^2_{ik}-\tensor{{W^M}}{_i_{\mu'}_k^{\mu'}} \right)-\frac{h_{ik}}{2k(k-1)}\left(|H|^2 - |B|^2\right) \\& - R^M h_{ik} \frac{1}{2n(n-1)} + R^M_{ik}\frac{1}{n-1} +\frac{h_{ik}}{2k(k-1)} \tensor{{W^M}}{_{a}_{\mu'}^{a}^{\mu'}}.
\end{split}
\]
That is,
\begin{equation}\label{FialkowGauss1}
\begin{split}
    P^Y_{ik} =&\; P^M_{ik}+ \frac{1}{k-1}\left( \sum_{\eta'=1}^m B^{\eta'}_{ik} H^{\eta'} - \sum_{\eta'=1}^m ( B^{\eta'})^2_{ik}-\tensor{{W^M}}{_i_{\mu'}_k^{\mu'}} \right)-\frac{h_{ik}}{2k(k-1)}\left(|H|^2 - |B|^2\right) \\ & + \frac{h_{ik}}{2k(k-1)} \tensor{{W^M}}{_{a}_{\mu'}^{a}^{\mu'}}.
\end{split}
\end{equation}

Finally, let us rewrite everything in terms of the trace-free second fundamental form $\mathring{B}^{\eta'}_{ik} = B^{\eta'}_{ik} - \frac{1}{k+1}H^{\eta'}h_{ik}$. Notice that
\[
\begin{split}
|\mathring{B}|^2 &= \sum_{\eta'=1}^m  h^{ij} h^{kl}\left( B^{\eta'}_{ik} - \frac{1}{k+1}H^{\eta'}  h_{ik}\right)\left( B^{\eta'}_{jl} - \frac{1}{k+1}H^{\eta'}  h_{jl}\right) = |B|^2 + \frac{1}{k+1}|H|^2 - \frac{2}{k+1}|H|^2 \\&= |B|^2 - \frac{1}{k+1}|H|^2 = |B|^2 - |H|^2 + \frac{k}{k+1}|H|^2
\end{split}
\]
and 
\[
\begin{split}
(B^{\eta'})^2_{ik} &= h^{jl}B^{\eta'}_{il}B^{\eta'}_{jk} = h^{jl}\left(\mathring{B}^{\eta'}_{il}+\frac{1}{k+1}H^{\eta'}h_{il}\right)\left(\mathring{B}^{\eta'}_{jk}+\frac{1}{k+1}H^{\eta'}h_{jk}\right) \\ &= (\mathring{B}^{\eta'})^2_{ik}+ \frac{2}{k+1} H^{\eta'}\mathring{B}^{\eta'}_{ik} + \frac{1}{(k+1)^2}(H^{\eta'})^2h_{ik}
\end{split}
\]
Putting everything together and using (\ref{FialkowGauss1}) we obtain 
\[
\begin{split}
    P^Y_{ik} - P^M_{ik} =&\; \frac{1}{k-1}\sum_{\eta'=1}^m \left(\mathring{B}^{\eta'}_{ik}+\frac{1}{k+1}H^{\eta'}h_{ik}\right)H^{\eta'} - \frac{1}{k-1}\sum_{\eta'}^m\tensor{{W^M}}{_i_{\mu'}_k^{\mu'}} \\ &- \frac{1}{k-1}\sum_{\eta'=1}^m\left((\mathring{B}^{\eta'})^2_{ik} + \frac{2}{k+1}H^{\eta'}\mathring{B}^{\eta'}_{ik} + \frac{1}{(k+1)^2}(H^{\eta'})^2h_{ik}\right) \\&-\frac{h_{ik}}{2k(k-1)}\left(\frac{k}{k+1}|H|^2 - |\mathring{B}|^2\right)+  \frac{h_{ik}}{2k(k-1)} \tensor{{W^M}}{_{a}_{\mu'}^{a}^{\mu'}}
    \\=&\; \frac{1}{k-1}\left(\left(1-\frac{2}{k+1}\right)\sum_{\eta'=1}^m\mathring{B}^{\eta'}_{ik}H^{\eta'} + h_{ik}\left(\frac{1}{k+1} - \frac{1}{(k+1)^2}\right)|H|^2 - \sum_{\eta'=1}^m(\mathring{B}^{\eta'})^2_{ik}\right) \\ &-\frac{1}{k-1}\tensor{{W^M}}{_i_{\mu'}_k^{\mu'}} - \frac{h_{ik}}{2k(k-1)}\left(\frac{k}{k+1}|H|^2 - |\mathring{B}|^2\right) + \frac{h_{ik}}{2k(k-1)} \tensor{{W^M}}{_{a}_{\mu'}^{a}^{\mu'}} \\=&\; \frac{1}{k-1}\left(\frac{k-1}{k+1}\sum_{\eta'=1}^m\mathring{B}^{\eta'}_{ik}H^{\eta'} +h_{ik}\frac{k}{(k+1)^2}|H|^2 -\sum_{\eta'=1}^m (\mathring{B}^{\eta'})^2_{ik}\right)\\ &-\frac{1}{k-1}\tensor{{W^M}}{_i_{\mu'}_k^{\mu'}} - \frac{h_{ik}}{2(k-1)(k+1)}|H|^2 + \frac{h_{ik}}{2k(k-1)}|\mathring{B}|^2 + \frac{h_{ik}}{2k(k-1)} \tensor{{W^M}}{_{a}_{\mu'}^{a}^{\mu'}}
    \end{split}
    \]
  That is,   
    \begin{equation}\label{Fialkow--GaussEq2}
    \begin{split}
    P^Y_{ik} - P^M_{ik} =&\; \frac{1}{k+1}\sum_{\eta'=1}^m\mathring{B}^{\eta'}_{ik}H^{\eta'} + \frac{h_{ik}}{2(k+1)^2}|H|^2 - \frac{1}{k-1}\sum_{\eta'=1}^m(\mathring{B}^{\eta'})^2_{ik}\\ &-\frac{1}{k-1}\tensor{{W^M}}{_i_{\mu'}_k^{\mu'}}  + \frac{h_{ik}}{2k(k-1)}|\mathring{B}|^2 + \frac{h_{ik}}{2k(k-1)} \tensor{{W^M}}{_{a}_{\mu'}^{a}^{\mu'}}.
\end{split}
\end{equation}
As in \cite{BlitzSamuel2021CFFa} (see (3.5)), we shall call relation (\ref{FialkowGauss1}) and (\ref{Fialkow--GaussEq2}) the Fialkow--Gauss equations.

\newpage
\section{Declarations}

\subsection{Data Availability Statement}  Our manuscript has no associated data.

\subsection{Funding and/or Conflicts of interests/Competing interests} The authors did not receive support from any organization for the submitted work. The authors have no competing interests to declare that are relevant to the content of this article.
\subsection{A.I. Disclosure}
No AI tool was used to generate mathematical results, proofs, narrative, or manuscript text, nor was it used to conduct literature searches. The authors assume responsibility for all content. OpenAI Codex (accessed October 2026) was used solely to assist in translating sketches by the authors into TikZ code for figures. The authors reviewed, edited the resulting code, and verified that the final figures accurately represent the intended pictures. 
\newpage
\bibliographystyle{amsplain}
\bibliography{bibliography.bib}

\end{document}